\documentclass[a4paper,11pt,english]{smfart}
\usepackage[cm]{aeguill}
\usepackage{amsfonts}
\usepackage{latexsym}
\usepackage[dvips]{graphicx}
\usepackage{color}
\definecolor{marin}{rgb}   {0.,   0.3,   0.7} 
\definecolor{rouge}{rgb}   {0.8,   0.,   0.} 
\definecolor{sepia}{rgb}   {0.8,   0.5,   0.} 
\usepackage[colorlinks,citecolor=marin,linkcolor=rouge,
            bookmarksopen,
            bookmarksnumbered
           ]{hyperref}

\usepackage{amsmath}
\usepackage{amssymb}
\usepackage{amsthm}
\usepackage{bbm}
\usepackage{nicefrac}
\usepackage[ruled, vlined]{algorithm2e}

\usepackage[applemac]{inputenc}

\usepackage[T1]{fontenc}
 
\newtheorem{df}{Definition}[section]
\newtheorem{lm}[df]{Lemma}
\newtheorem{pr}[df]{Proposition}
\newtheorem{Th}[df]{Theorem}

\newtheorem{rem}[df]{Remark}

\newtheorem*{hyp1}{Hypothesis (i)}
\newtheorem*{hyp2}{Hypothesis (ii)}
\newtheorem*{hyp3}{Hypothesis (iii)}

\newcommand{\e}{h}

\newcommand{\R}{\mathbb{R}}

\newcommand{\To}{\mathbb{T}}

\newcommand{\Tc}{\mathcal{T}}

\renewcommand{\AA}{\widetilde{\mathcal{A}}}

\renewcommand{\d}{\mathrm{d}}

\newcommand{\conv}{*}
\newcommand{\cO}{\mathcal{O}}
\newcommand{\N}{\mathbb{N}}

\newcommand{\SNorm}[2]{|#1|\left.\vphantom{T_{j_0}^0}\!\!\right._{#2}}          

\author{Anne Bouillard}
\address{Ecole Normale Sup\'erieure, 45 rue d'Ulm, 75230 Paris Cedex 05, France.}
\email{anne.bouillard@ens.fr}

\author{Erwan Faou}
\address{INRIA \& Ecole Normale Sup\'erieure de Cachan Bretagne, 
Avenue Robert Schumann 35170 Bruz, France } 
\email{Erwan.Faou@inria.fr}

\author{Maxime Zavidovique}
\address{Institut de Math\'ematiques de Jussieu, 
Universit\'e Pierre et Marie Curie, Case 247 
4, place Jussieu, 
75252 Paris Cedex 05, France}
\email{Maxime.Zavidovique@upmc.fr }

\title[Fast weak--KAM integrators]{
Fast weak--KAM integrators for separable Hamiltonian systems 
}

\begin{document}

\begin{abstract}
We consider a numerical scheme for Hamilton--Jacobi equations based on a direct discretization of the Lax--Oleinik semi--group. We prove that this method is convergent with respect to the time and space stepsizes provided the solution is Lipschitz, and give an error estimate. Moreover, we prove that the numerical scheme is a {\em geometric integrator} satisfying a discrete weak--KAM theorem which allows to control its long time behavior. Taking advantage of a fast algorithm for computing min--plus convolutions based on the decomposition of the function into concave and convex parts, we show that the numerical scheme can be implemented in a very efficient way.  

\end{abstract}

\thanks{
}
\maketitle


\section{Introduction}

We consider Hamilton--Jacobi equations of the form
\begin{equation}
\label{eq:hjb}
\partial_t u + H(t,x,\nabla u) = 0, \quad u(0,x) = u_0(x), 
\end{equation}
where $H(t,x,v)$ is a Hamiltonian function   $H: \R \times \R^n \times \R^n \to \R$ that 
is {\bf separable}, in the sense that we can write
\begin{equation}
\label{eq:HKV}
H(t,x,p) = K(p) + V(t,x), 
\end{equation}
for some convex function $K$ and some smooth and bounded function $V$.   The typical cases of study we have in mind  are the so called mechanical Hamiltonians, of the form
\begin{equation}
\label{eq:HPV}
H(t,x,p) = \frac{1}{2}\SNorm{p + P}{}^2 + V(t,x)
\end{equation}
where $P \in \R^n$ is a given vector,  $\SNorm{v}{}^2 = v_1^2 + \cdots + v_n^2$ for $v = (v_1,\cdots, v_n) \in \R^n$, and where $V(t,x)$ is a suitably smooth and bounded function. 

Since the pioneering works of Crandall and Lions \cite{Crandall} and Souganidis \cite{souganidis}, the study of numerical schemes for the Hamilton--Jacobi equation \eqref{eq:hjb} has known many recent progresses, see for instance \cite{lionssouganidis, lintadmor, abgrall, jinxin, barlesjakobsen} and the references therein, and more specifically  
 \cite{oshersethian, oshershu, jiangpeng} for the popular (weighted) essentially nonoscillatory (ENO and WENO) methods which are now widely commonly used in many application fields.  Let us finally mention a recent work by Soga (\cite{soga}) which deals with situations similar to the ones tackled in this paper.

Following a different approach, more in the spirit of \cite{falcone, rorro} (or \cite{akian} in an optimal control setting), the main aim of this paper is to show how a direct discretization of the Lax--Oleinik representation of the viscosity solution of \eqref{eq:hjb} allows to define a new fast algorithm for computing $u(t,x)$ possessing strong geometrical properties allowing to control its long time behavior and obtain error estimates when the solution is Lipschitz.

Let us recall, see \cite{Lions,Fathi}, that under some assumptions on $H$ (smoothness, uniform superlinearity and  strict convexity over the fibers, see  Section \ref{S2} below), we can write
\begin{equation}
\label{eq:E1}
u(t,x) = \inf_{\gamma(t) = x} u_0\big(\gamma(0)\big)+\int_{0}^{t} L\big(s,\gamma(s),\dot{\gamma}(s)\big)\d s, 
\end{equation}
where the infimum is taken over all absolutely continuous curves $\gamma: [0,t] \to \R^n$ such that $\gamma(t) = x$, and where $L(t,x,v)$ is the Lagrangian associated with $H$. The idea of this paper is to discretize directly \eqref{eq:E1} on a space--time grid, by replacing the set of curves $\gamma$ by the set of piecewise linear  (or piecewise constant)  curves across the space grid points.

We first prove that such an approximation is convergent with respect to the size of the space and time stepsizes, and under an anti--CFL condition (namely that the ratio between the space and time stepsize should be small). We give an error estimate under the assumption that $u_0$ is Lipschitz. 

Moreover, this numerical integrator turns out to be a {\em geometric integrator} (see for instance \cite{hairer06gni,leimreich}) in the sense that it respects the long time behavior of the exact solution $u(t,x)$. Let us recall that in the case of periodic  Hamiltonians (both in time and space variables), the weak--KAM theorem (see \cite{Fathi,cis}) shows the existence of a constant $\overline H$ such that 
$$
\frac{1}{t}u(t,x) \to \overline H \quad \mbox{when}\quad t \to +\infty.
$$
Here, using a discrete weak--KAM theorem, see \cite{Be,Za}, we prove that the numerical scheme possesses the same long time property, with a constant that is close to the exact constant $\overline H$. 

Finally, we show that in the separable case mainly considered in this paper (see \eqref{eq:HKV} below), the discrete version of \eqref{eq:E1} is a min--plus convolution that can be approximated using a fast algorithm with $\mathcal{O}(N)$ operations in many situations if $N$ is the number of grid points. This algorithm uses the decomposition of $u$ into concave and convex parts. Moreover, it easily extends to any space dimension $n$ using a splitting strategy, when the kinetic part of the Hamiltonian is separable  -- see Remark \ref{rk28} -- which includes the case \eqref{eq:HPV}. 

We then conclude by numerical simulations in dimension 1 to illustrate the good behavior of our algorithm, as well as its very low cost in general situations.

The paper is divided into three parts: in a first part (Section \ref{S2}) we give a convergence result over a finite time interval of the form $[0,T]$ where $T$ is fixed. In a second part (Section \ref{S3}), we consider the case where the Hamiltonian is periodic in time $t$ and $x$. In this case, we can derive explicitly the dependence on $T$ in the error estimates, and prove a weak--KAM theorem for the numerical scheme which gives informations concerning the long time behavior of the scheme. In the third part (Section \ref{S4}), we describe the implementation of the method based on a fast algorithm to compute min--plus convolutions. We conclude this part by showing numerical simulations.

\subsection*{Acknowledgement} This work owes a lot to Vincent Calvez, who put the authors in touch and took part to preliminary discussions. It is a great pleasure to thank him a lot. We also would like to thank Vinh Nguyen for careful reading through previous versions of the paper. The last author would like to thank Antonio Siconolfi for bringing him to this subject.  Finally, we thank the anonymous referees for very helpful remarks on improving the content and presentation of this manuscript.

\section{Description of the scheme and convergence results\label{S2}}

\subsection{Hypotheses}

Recall that we consider a  separable Hamiltonian $H(t,x,p)$  of the form \eqref{eq:HKV}. 
With this Hamiltonian we can associate by Legendre transform the Lagrangian 
$$
L(t,x,v) = \sup_{p \in \R^n} \Big(p \cdot v - H(t,x,p)\Big), 
$$ 
and we compute that in our case, 
$$
L(t,x,v) = K^*(v) - V(t,x),
$$ 
where $K^*(v)$ is the Legendre transform of $K$. For instance in the special case \eqref{eq:HPV} we have 
$$
L(t,x,v) = \frac{1}{2}\SNorm{v }{}^2 - P \cdot v - V(t,x). 
$$
We make the following assumptions on $K$ and $V$: 

\begin{hyp1}
The function $K^* \in \mathcal{C}^2 (\R^n)$ is uniformly strictly convex in the sense that there exists a constant $c>0$ such that for all $Y \in \R^n$, and for all $v \in  \R^n$, 
\begin{equation}
\label{eq:conv}
\frac{\partial^2 K^*}{\partial v^2}(v)(Y,Y) \geqslant c \SNorm{Y}{}^2.
\end{equation}
\end{hyp1}

\begin{hyp2}
The function $V(t,x) \in \mathcal{C}^2(\R \times \R^n)$ is such that there exists a constant $B$ such that for  $j+ q \leqslant 2$, and all $(t,x) \in \R \times  \R^n$, 
\begin{equation}
\label{eq:hyp1}
|\partial_t^j\partial_x^q  V(t,x)|\leqslant B,
\end{equation}
where $|\cdot|$ denote the norm of differential operators acting on $\R \times \R^n$. 
\end{hyp2}

Note that the bound \eqref{eq:hyp1} is straightforward under the additional assumption that $V(t,x,v)$ is {\em periodic} in $(t,x)$, the case studied in the next section. 

\begin{rem}\rm
The previous hypotheses imply that the Hamiltonian $H$ and the Lagrangian $L$ are $\mathcal C^2$, convex  and superlinear in respectively $p$ and $v$:
\begin{equation}
\label{eq:suplin}
\forall\,  k>0, \quad  \forall\, t >0, \quad \exists \, A(k) < \infty, \quad L(t,x,v)\geqslant k|v|-A(k).
\end{equation}

\end{rem}


Under these assumptions, the viscosity solution of \eqref{eq:hjb} can be represented by the formula: for all $t,\delta > 0$, 
\begin{equation}
\label{eq:laxol}
\forall x\in \R^{n}, \quad u(t+ \delta,x) =\inf_{\gamma(t+ \delta) = x} u\big(t,\gamma(t)\big)+\int_{t}^{t + \delta} L\big(s,\gamma(s),\dot{\gamma}(s)\big)\d s,
\end{equation}
where the infimum is taken on all absolutely continuous curves $\gamma:(t,t+ \delta) \to \mathbb{R}^n$ verifying $\gamma(t+ \delta)=x$, see \cite{Lions,Fathi}. We will later on use the notation $u(t+ \delta,x) :=  T^{\delta}_tu(x)$. Moreover, the infimum is achieved on a curve $\gamma_{t,x}^\delta(s)$ that is $\mathcal{C}^2$ and satisfies the Euler--Lagrange equation
\begin{equation}
\label{eq:Eulag}
\frac{\d}{\d s} \frac{\partial L}{\partial v}\big(s,\gamma(s),\dot \gamma(s)\big) = \frac{\partial L}{\partial x} \big(s,\gamma(s),\dot \gamma(s)\big).
\end{equation}
The notation $T_t^\delta$ defines the Lax--Oleinik semi--group. In particular, we have $T_{t+\delta}^{\sigma} \circ T_{t}^{\delta} = T_{t}^{\delta + \sigma}$ for non negative $\delta$ and $\sigma$. 
With these assumptions, we have the following Proposition.
\begin{pr}
\label{lm1}
For all $T > 0$, and for all $R > 0$, there exists $M(R,T)$ such that for all $x, y \in \R^n$ satisfying $\SNorm{x - y}{}  \leqslant R$ and for all $t \in \R$, then every solution of the Euler--Lagrange equation \eqref{eq:Eulag} minimizing the action 
\begin{equation}
\label{eq:action}
\int_{t}^{t+T} L\big(s,\gamma(s), \dot \gamma(s)\big) \d s,
\end{equation}
with fixed endpoints
 $\gamma(t) = x$ and $\gamma(t + T) = y$, satisfies 
$\SNorm{\dot \gamma(s)}{}  \leqslant M(R,T)$ for all $s \in [t,t + T]$. 
\end{pr}

\begin{proof}
The Euler--Lagrange equation is written 
$$
\frac{\partial^2 K^*}{\partial v^2}\big(\dot\gamma(s)\big) \big(\ddot \gamma(s)\big) = - \frac{\partial V}{\partial x}\big(s,\gamma(s)\big). 
$$
Using the uniform strict convexity of $K^*$ and the fact that $\partial_x V$ is uniformly bounded, there exists a constant $C$ depending only on $T$ and $K$, such that 
\begin{equation}
\label{eq:2deriv}
\forall \, s  \in [t,t + T]\quad\SNorm{\ddot \gamma(s)}{}^2 \leqslant C. 
\end{equation}
This implies that for all $s \in [t,t + T]$, 
\begin{equation}
\label{eq:gam}
| \dot \gamma(s) - \dot \gamma(t) |Ê \leqslant \int_{t}^{t + T} \SNorm{\ddot \gamma(s)}{} \d s  \leqslant T \sqrt{C}.
\end{equation}
Now as $\gamma$ minimizes the action between $t$ and $t + T$, comparing with the trivial curve $t \mapsto x + t(y-x)/T$ from $x$ to $y$, we get
\begin{eqnarray*}
\int_{t}^{t+T} L\big(s,\gamma(s), \dot \gamma(s)\big)  \d s &\leqslant& \int_{t}^{t+T} K^*\big(\frac{y-x}{T}\big) - V\big(s,x + \frac{t}{T}(y-x)\big) \d s\\
&\leqslant& T D(R/T) + TB,
\end{eqnarray*}
where $D(M) = \sup\limits_{|v| \leqslant M}| K^*(v)|$, and $B$ is given by \eqref{eq:hyp1}. 
By superlinearity, we deduce that 
$$
\int_{t}^{t+T} \SNorm{\dot \gamma(s)}{} \d s  \leqslant T D(R/T) + TA(1)+  TB,
$$
where $A(1)$ is given by \eqref{eq:suplin}. 
Using \eqref{eq:gam}, we thus obtain
\begin{eqnarray*}
T |\dot \gamma(t)|  \leqslant \int_t^{t+T}|\dot \gamma(s)| \d s + T^2 \sqrt{C}  \leqslant T D(R/T) + TA(1)+  TB + T^2 \sqrt{C}. 
\end{eqnarray*}
This shows that $|\dot\gamma(t)|$ is bounded, and hence using again \eqref{eq:gam} that $|\dot \gamma(s)|$ is bounded for all $s$, with a constant depending only on $T$, $R$ and the constants appearing in \eqref{eq:conv} and \eqref{eq:hyp1}. 
\end{proof}

\begin{rem}\rm
The previous lemma is one of the main keys in the proof of the convergence of our schemes  (compare with \cite{falcone2,falcone}).   Here, it is established thanks to the particular form of $H$, but it can be noted that it remains valid under other technical assumptions (for example if $H$ is autonomous and Tonelli as established in \cite{Fathi,FaMa},
 or if it is Tonelli and periodic both in the space and in the time variable as proven in \cite{Ma2,cis,itu}). Actually, in these cases, it can be established that $M(R,T)$ only depends on the ratio $R/T$. We will come back on these matters  in Section \ref{S3} and  give a proof of this result in the Appendix. 
Therefore, this section and the  next  would still be valid for general Hamiltonians chosen in these classes, however the convolution techniques of section \ref{S4} would fail.

Finally, a clear consequence of Equation \eqref{eq:2deriv} is that the Euler--Lagrange flow of $L$ is complete.
\end{rem}

\subsection{{A first discrete semi--group}}

 For a given $\e>0$ we define the $\e$-grid $G_{\e}=\e \mathbb{Z}^n$ endowed with the metric induced by the euclidian metric on $\R^{n}$. For a given continuous function $u$, we define $u|_{G_{\e}}: G_{\e} \to \R$ its restriction to the grid $G_{\e}$. 
 
 A first idea to discretize the Lax--Oleinik semi--group is as follows. Given $t,\tau>0$, let us define $c_{t,\e}^\tau:G_{\e}^2\mapsto \R{}$ as follows:
\begin{equation}
\label{eq:cost1}
\forall (x,y)\in G_{\e}^2,\quad  c_{t,\e}^\tau(x,y)=\int_t^{t+\tau} L\Big(s,x+(s-t)\frac{y-x}{\tau},\frac{y-x}{\tau}\Big)\d s.
\end{equation}

%
%

Let us introduce the following discrete Lax--Oleinik semi--group: if $u: G_{\e}\to \mathbb{R}$ is any function, we set 
$$
\forall x\in G_{\e} , \quad T_{{t,\e}}^{\tau} u (x)=\inf_{y\in G_{\e}} u(y)+c_{t,\e}^\tau(y,x).
$$


For a given integer $N$, we may define 
\begin{equation}
\label{eq:compo}
T_{{t,\e}}^{N\tau} u = T_{{t_{N-1},\e}}^{\tau} \circ \cdots \circ T_{t_1,\e}^\tau  \circ T_{{t_0,\e}}^{\tau} u
\end{equation}
the composition of $N$ times the discrete semi--group $T_{{t,\e}}^{\tau}$, where for all $i = 1,\ldots,N-1$, $t_i = t + i \tau$. 



 One of the nice features of this discretization is the following property:
 let $u:\R^{n}\mapsto \R{}$ and $N \geqslant 1$ an integer, $t \in \R$ and $\tau>0$. Then 
 \begin{equation}
 \label{eq:up}
 (T^{N\tau}_t u)|_{G_{\e}}\leqslant T_{t,\e}^{N\tau} (u|_{G_{\e}}).
\end{equation}
Indeed this semi--group consists in taking an infimum over a smaller set of curves, compared to the Lax--Oleinik semi--group. 

\subsection{Fully discrete semi--group}

 The main disadvantage is that   to compute this cost, a quadrature rule in time has to be used. In this subsection, we prove how the Euler approximation of this integral yields a convergent scheme which still  satisfies a weak--KAM theorem similar to Proposition \ref{P:weakKAM} under suitable periodicity assumptions. 

For a given $\e>0$ and $\tau>0$, we define the following cost function: 
\begin{equation}
\label{eq:cost2}
\forall (x,y)\in G_{\e}^2,\quad  \kappa_{t,\e}^{\tau}(y,x)= \tau  L\Big(t,x,\frac{x-y}{\tau}\Big). 
\end{equation}
and we introduce the associated fully discrete Lax--Oleinik semi--group, which acts on any function  $u: G_{\e}\to \mathbb{R}$ as follows:
$$
\forall x\in G_{\e} , \quad \Tc_{{t,\e}}^{\tau} u (x)=\inf_{y\in G_{\e}} u(y)+\kappa_{t,\e}^\tau(y,x).
$$
 Using the explicit expression of $L$, we can rewrite this fully-discrete semi--group as 
\begin{equation}
\label{eq:eqconvol}
\forall x\in G_{\e} , \quad \Tc_{{t,\e}}^{\tau} u (x)=  \inf_{y\in G_{\e}} \Big( u(y)+ \tau K^*\big(\frac{x-y}{\tau}\big) \Big) -\tau V(t,x).  
\end{equation}
involving the (min,plus)--convolution of $u$ and $K^*$. 

\begin{rem}\label{SLlike}\rm
 This scheme is a particular case of the so called Semi--Lagrangian schemes. Indeed, those are of the form
$$U(t+\tau,x)=\inf_{\alpha\in \mathcal A} U(t,x-h\alpha)+\tau L(x,\alpha),$$
where $\mathcal A$ is a given (usually compact) set of controls. Here, we take $\alpha = \frac{x-y}{\tau}$. Note that the main feature of this choice is that it enables $y$ to remain on the grid, whereas standard semi-Lagrangian methods rely on an interpolation procedure at each step. This particular form allows to use the fast convolution techniques of Section \ref{S4}.
\end{rem}

\begin{rem}\rm
We can interpret this scheme as a discretization of the splitting scheme (see for instance \cite{splitting}) with time step $\tau$ based on the decomposition 
$$
\partial_t u(t,x) + K\big(\nabla u(t,x)\big) = 0 ,\quad \mbox{and} \quad \partial_t u(t,x) + V(t,x)  = 0, 
$$
where the first part is integrated using the method described in the previous section.
\end{rem}

\begin{rem}
\label{rk28}
\rm 
In dimension $n \geqslant 1$, if we assume that for $p = (p_1,\ldots,p_n) \in \R^n$, $K(p) = K_1(p_1) + \cdots + K_n(p_n)$ with convex Hamiltonian functions $K_i^*$, $i = 1,\ldots,n$, satisfying all the hypotheses {\it\textbf{{(i)}}}, {\it\textbf{{(ii)}}} on $\R$, then we immediately see that for a given function $u(x) = u(x_1,\ldots,x_n)$, with $x = (x_1,\ldots,x_n) \in \R^n$,  we have 
\begin{multline*}
\inf_{y \in G_{\e}} u(y) + \tau K^*\big(\frac{x- y}{\tau}\big) = \\
\inf_{y_n \in G^n_\e} \left[  \tau K_n^*\big(\frac{x_n - y_n}{\tau}\big) + \Big[ \inf_{y_{n-1} \in G^{n-1}_\e}\tau K_{n-1}^*\big(\frac{x_{n-1} - y_{n-1}}{\tau}\big)  + \cdots  \right.\\
 \left.+  \Big[\inf_{y_1 \in G^1_\e}  \tau K_n^*\big(\frac{x_1 - y_1}{\tau}\big) + u(y_1,\ldots,y_n) \Big] \cdots \Big]\right]=\\
\Tc_{{t,\e}}^{\tau,1}\circ \cdots \circ \Tc_{{t,\e}}^{\tau,n} u (x), 
\end{multline*}
where we have decomposed $G_{\e} = G^1_\e \times \cdots \times G^n_\e$ and where 
$$\forall i \in [1,n], \quad \Tc_{{t,\e}}^{\tau,i} u (x)=\inf_{y_i \in G^i_\e} \tau K_i^*\big(\frac{x_i - y_i}{\tau}\big) + u(x_1,\ldots,x_{i-1},y_i,x_{i+1},\ldots x_n) .$$
  %
  This formula is essentially due to the fact that the Hamiltonians $K_i$ commute, i.e. satisfy $\{ K_i,K_j\} = 0$ for $(i,j) \in \{1,\ldots,n\}^2$ which ensures that the flows of 
$\partial_t u = K_i(\nabla u)$, $ i = 1,\ldots,n$ commute. 
In this case, this allows to reduce the computation of the minimum over the $n$ dimensional grid $G_{\e}$ to $n$ minimization problems over the one--dimensional  grids $G_{\e}^i$. 

\end{rem}

For a given integer $N$,  and a given function $u^0:G_h \to \R$  we define -- compare \eqref{eq:compo}
\begin{equation}
\label{eq:uN}
\forall\, x \in G_h, \quad u^N(x) := \Tc_{{t,\e}}^{N\tau} u(x) = \Tc_{{t_{N-1},\e}}^{\tau} \circ \cdots \circ \Tc_{t_1,\e}^\tau  \circ \Tc_{{t_0,\e}}^{\tau} u^0(x)
\end{equation}
where $t_i = t + i \tau$.   When no confusion is possible, we will also refer to $\Tc^N$ instead of $\Tc_{{t,\e}}^{N\tau}$.   Note that with these notations, an estimate of the form \eqref{eq:up} is no longer valid.
\begin{rem}\rm
The following monotonicity property holds true: if $u(x)\leqslant v(x)$ for all  $x \in G_h$,  we easily observe that $ \Tc^N u \leqslant \Tc^N v$.
%
\end{rem}

Finally, we have the following convergence result:

\begin{Th}\label{approx2}
 Let $T > 0$,  $\tau_0$ and $h_0 > 0$ and $u_0: \R^n \to \R$ a bounded Lipschitz function.  For a given $t_0\in \R$, let $u(t,x) =T^{t}_{t_0} u $ be the viscosity solution \eqref{eq:E1} of the Hamilton--Jacobi equation \eqref{eq:hjb} such that $u(t_0,x) = u_0$. Then there exists a constant $M(T)$ such that for  all $\e>0$ and $\tau>0$ such that $\e < \e_0$, $\tau < \tau_0$ and the bound 
 \begin{equation} \label{eq:anticfl}
 \frac \e \tau < h_0
 \end{equation}
  are satisfied, 
then for all $N$ verifying $N \tau \leqslant T$, 
\begin{equation}\label{eq:conver}
\forall\, x \in G_{\e}, \quad  |u(N\tau,x) - u^N(x)|  \leqslant M(T) \big(\frac{\e}{\tau}+\tau\big). 
\end{equation}
where the discrete solution $u^N:G_{\e} \to \R$ is given by the formula \eqref{eq:uN} with initial value $u^0 := u_0 |_{G_{\e}}$.  
\end{Th}

\begin{rem}\rm
The following estimate is not new and not surprising. First, since the scheme is monotonous, its convergence is known to be almost automatic (see \cite{BaSo}). Moreover, as we deal with  a particular case of Semi--Lagrangian scheme, such results are known (see \cite{falcone2,falcone} and references therein), in particular the convergence in $\cO ( \e/\tau) + \cO ( \tau)$ is common. For the sake of completeness, we will give a proof in the appendix \ref{BB}. 
\end{rem}

\begin{rem}\rm
In the previous theorem, taking $\tau = \sqrt{\e}$ obviously yields a convergence of order $\cO (\sqrt \e)$. Finally, note that in the present setting, the constant $M$  depends on $T$ in an uncontrolled way. However, under some extra periodicity assumptions, this dependance will become linear later in the paper (see Proposition \ref{approx3}).
\end{rem}

%

\section{Long time behavior in the periodic case\label{S3}}

We will now make the additional assumption that the potential function $V(t,x)$ is periodic, namely
\begin{hyp3}
 The function $V$ is $\mathbb{Z}\times \mathbb Z^n$-periodic, in the sense that
$$\forall (t,x)\in \R \times \R^{n},\quad \forall (m,M)\in \mathbb Z \times \mathbb Z^n,\quad V(t,x)=V(t+m,x+M).$$
\end{hyp3}

Note that under this assumption, the estimate \eqref{eq:hyp1} is automatically satisfied since we still assume that $V \in \mathcal{C}^2(\R \times \R^n)$. 

\subsection{Weak--KAM theorem and a priori compactness}

In this  periodic case, the weak--KAM theorem allows to study the long time behavior of the solution of \eqref{eq:hjb} defined by the Lax--Oleinik semi--group: 
\begin{pr}
\label{P:weakKAM}
Assume that the hypotheses $\textbf{{(i)}}$ and $\textbf{{(iii)}}$ are satisfied. 
 Then there exists a unique constant $\overline H$ for which the equation
$$
\forall t>0, \quad T_t^1 u^*(t,\cdot) = \overline H + u^*(t,\cdot),
$$
admits a $\mathbb{Z}\times \mathbb Z^n$--periodic continuous solution: $u^*: \R\times \R^n \to \R$.

Moreover, for any uniformly bounded $u:\R^n\to \R$, there exists a constant $C_u$ such that
 $$
\forall\, t > 0, \quad |T_0^t u - t \overline H  |_\infty  \leqslant C_u. 
$$
\end{pr}

\begin{proof}
 This result is very standard and a complete proof can be found in \cite{Fathi}. 
The existence of $\overline H$ and of $u^*$ is exactly the content of the weak--KAM theorem (see \cite{FaKAM,Fathi} for the autonomous case, and \cite{cis} for the time periodic case). The second assertion is a consequence of the fact that 
$$
\SNorm{T_0^t u  - T_0^tv}{\infty} \leqslant \SNorm{u-v}{\infty}
$$
for all continuous bounded functions $u$ and $v$ on $\R^n$, where $|\cdot |_{\infty}$ denotes the $L^\infty$ norm on $\R^n$. 
\end{proof}

The goal of this section is to prove that a similar result holds for the numerical scheme described in the previous section, under the hypotheses \textbf{{\em (i)}} and \textbf{{\em (iii)}}.


In order to study the long time behavior of the method in this case, we first give an a priori compactness result which refines the estimates given in Proposition \ref{lm1}. 
The following proposition is mainly due to Mather (see \cite{Ma2} for the case of time periodic Lagrangians, or \cite[Lemma 7 and Corollary 8]{itu} for space periodic Lagrangians).
\begin{pr}\label{lm2}
Assume that $H$ satisfies the hypotheses \textbf{{(i)}} and \textbf{{(iii)}}. 
For all $\Gamma>0$, there exists a constant $\Gamma'$ such that for any minimizer of the Lagrangian action $\gamma : [a,b]\to \R^n$ with $b-a\geqslant 1$ and $|\gamma(b)-\gamma(a)|/(b-a)\leqslant \Gamma$ then we have
$$\forall t\in [a,b],\quad |\dot \gamma(t)|\leqslant \Gamma'.$$
\end{pr}
In other words, the constant $M(R,T)$ of Proposition \ref{lm1} can be chosen to be an increasing function of $R/T$. 

For the sake of completeness, we will give in appendix a proof of this proposition. Note that most of the proof -- essentially taken from \cite{Ma2} -- does not require  the Hamiltonian to be periodic in space.  
In \cite{itu} a similar result is proven  which requires the Lagrangian to be periodic in space, but not any more in time.

\subsection{Convergence estimates in the periodic case}

Using the previous proposition, we can compute explicitly the time dependence of the  constant $M(T)$ in the error estimates of Theorem  \ref{approx2}, and prove that it depends linearly on the time in the periodic case.   (see the Appendix \ref{BB} for a proof). 

\begin{Th}\label{approx3}
 Let $T_0 > 1$, $\tau_0$ and $h_0 > 0$ and $u_0: \R^n \to \R$ a bounded Lipschitz function.  For a given $t_0\in \R$, let $u(t,x) =T^{t}_{t_0} u $ be the viscosity solution \eqref{eq:E1} of the Hamilton--Jacobi equation \eqref{eq:hjb} such that $u(t_0,x) = u_0$.  y Then, there exists a constant $M$ such that for all $\e>0$ and $\tau>0$ such that $\e < \e_0$, $\tau < \tau_0$ and 
 \begin{equation}
 \frac{ \e}{\tau} < h_0, 
 \end{equation}
and for all $N$ satisfying $N \tau \geqslant T_0$, 
\begin{equation}
 \forall\, x \in G_{\e},\quad |u(N\tau,x) - u^N(x)| \leqslant  M N\tau\big(\frac{\e}{\tau}+\tau\big).
\end{equation}
where the discrete solution $u^N:G_{\e} \to \R$ is given by the formula \eqref{eq:uN} with initial value $u^0 := u_0 |_{G_{\e}}$.  
\end{Th}

\subsection{Discrete weak--KAM theorem and effective Hamiltonian}

Recall that the function $u(t,x)$ is defined on $\mathbb T^1\times \mathbb T^n = \big(\R / \mathbb Z \big)\times \big( \mathbb R^n / \mathbb Z ^n \big)$. For convenience, we will only treat the cases of rational time and space discretizations: We set
$$
\Lambda = \Big\{ \Big( \frac{1}{k},\frac{1}{\ell}\Big)\, | \, (k,\ell) \in \N^* \times \N^*\Big\},  
$$
and in the sequel, we will only consider stepsizes $(\e,\tau) \in \Lambda$. 
We then will denote by $p$ the canonical projection from $\R ^n$ to $\mathbb T^n$, and by $\widetilde G_{\e} = G_{\e}/\mathbb Z^n$ the quotiented grid, where $G_{\e}$ is the grid above, defined on $\R^n$. 

Finally, we define a new cost function: For $(\e,\tau) \in \Lambda$ and $t > 0$, 
$$
\forall \, (\tilde x,\tilde y)\in (\widetilde G_{\e})^2,\quad \tilde \kappa_{t,\e}^{\tau}(\tilde x,\tilde y)=\inf_{\substack{p(x)=\tilde x \\ p(y)=\tilde y}}  \kappa_{t,\e}^{\tau}(x,y), 
$$
where $\kappa_{t,\e}^{\tau}(x,y)$ is the fully discrete cost function defined in \eqref{eq:cost2}.

We then define the following semi--group:  Let $\tilde u : \mathbb T^n \to \R$ and $(\e,\tau) \in \Lambda$, then 
 the fully discrete semi--group is
\begin{equation}\label{eq:los}
\forall\,  \tilde x\in \widetilde G_{\e} , \quad \widetilde \Tc_{{t,\e}}^{\tau} \tilde u (\tilde x)=\inf_{\tilde y\in \widetilde G_{\e}} \tilde u(\tilde y)+\tilde \kappa_{t,\e}^{\tau}(\tilde y,\tilde x).
\end{equation}
As in the previous section, we define 
 $$
 \widetilde \Tc_{{t,\e}}^{N\tau}  =\widetilde \Tc_{{t_{N-1},\e}}^{\tau} \circ \cdots \circ \widetilde \Tc_{t_1,\e}^{\tau}  \circ \widetilde \Tc_{{t_0,\e}}^{\tau},$$
 where $t_i = t + i \tau$. 

\begin{rem}
\label{rk282}
\rm 
We leave to the reader the verification that if $\tilde u : \mathbb T^n \to \R$ is a function and if $u : \R^n \to \R$ is its lift (which is then $\mathbb{Z}^n$ periodic), the  function  $\Tc_{{t,\e}}^{N\tau} u $ is $\mathbb Z^n$ periodic and that the function it canonically induces on $\mathbb T^n$ is  $\widetilde \Tc_{{t,\e}}^{N\tau} \tilde u$. It comes from the fact that two infimums commute. Hence the previous convergence result Theorem \ref{approx3} can be read equivalently on $\To^n$ or on the space of $\mathbb{Z}^n$ periodic functions on $\R^n$. {However, it is easier to take advantage of the compactness of $\To^n$. This is why we introduce these new costs, and define them with infimums.}
\end{rem}
We can use the discrete weak--KAM theorem to better understand the approximate semi--groups applied  for a period $1$ of time and obtain the following proposition.
\begin{pr}
\label{prop:klm}
For any $(\e,\tau) \in \Lambda$,  there exists a  unique constant  $\overline{\mathcal  H}_{\e,\tau}$ such that there exists a function $  v_{\e,\tau}^* : \widetilde G_{\e}\to \R{}$ verifying:
$$
\widetilde \Tc_{{0,\e}}^{1}  v_{\e,\tau}^* = v_{\e,\tau}^* + \overline{\mathcal  H}_{\e,\tau}.
$$
Moreover, in $u$ is any bounded initial datum on $G_{\e}$ at $t = 0$, then we have in $L^\infty$
$$
\frac{1}{N\tau} \Tc_{{0,\e}}^{N\tau} u \longrightarrow  \overline{\mathcal H}_{\e,\tau}, 
$$
as $N \to +\infty$,  where $\Tc_{{0,\e}}^{N\tau}$ is defined in \eqref{eq:uN}.  
\end{pr}
\begin{proof}
The first part is just a reformulation of the discrete weak--KAM theorem (see for example the appendix of \cite{Za} or \cite{Be}) while the second part is -- as in the proof of Proposition \ref{P:weakKAM} --  a direct consequence of the fact that our approximation operators are weakly contracting for the infinity norm on bounded functions.
\end{proof}
\begin{rem}\rm
  Note that the previous proposition can be interpreted in the (min,plus)
  framework: Equation~(\ref{eq:los}) is the (min,plus) product of a
  vector $\tilde{u}$ by the matrix
  $\tilde{\kappa}_{t\, \e}^{\tau}$. Then, for $\tau = 1/\ell$, $\ell \in \N^*$, 
  $\widetilde{\Tc}_{0,\e}^{1}\tilde{u} = \widetilde{\Tc}_{0,\e}^{\ell \tau}\tilde{u}$ is obtained by
  successive matrix multiplications with $\tilde{u}$. Hence there exists a matrix $C_{\e,\tau}$ such
  that $\widetilde{\Tc}_{0,\e}^{1}\tilde{u}(\tilde{x}) =
  \inf\limits_{\tilde{y}\in\tilde{G}_{\e}}\tilde{u}(\tilde{y}) +
    C_{\e,\tau}(\tilde{y},\tilde{x})$, with $C_{\e,\tau}(\tilde{y},\tilde{x})<+\infty$ for
    all $\tilde{y},\tilde{x}$. The matrix $C_{\e,\tau}$ has a unique eigenvalue
    $\overline{\mathcal H}_{\e,\tau}$, and $v^*_{\e,\tau}$ is an eigenvector
    (see \cite{BCOQ} for details).
\end{rem}

Let us recall that $\overline{H}$ is the effective Hamiltonian of $H$. It is obtained in homogenization theory by solving the cell problem (\cite{LPV}) and is also the constant found in the weak--KAM theorem (\ref{P:weakKAM}). Using the refined convergence result obtained in Theorem \ref{approx3}, we can estimate the error between the effective Hamiltonian and the discrete effective Hamiltonian defined in Proposition \ref{prop:klm}. 

\begin{Th}
With the notations of Theorem \ref{approx3}, let $(\e,\tau) \in \Lambda$ be such that 
$\e \leqslant \e_0$, $\tau \leqslant \tau_0$ and $ \e/\tau \leqslant h_0$, then following inequality holds:
$$\left| \overline{\mathcal H}_{\e,\tau} - \overline H \right| \leqslant M\big(\frac{\e}{\tau} + \tau\big),$$
where $M$ is {the}   constant coming from Theorem \ref{approx3}, and where  $\mathcal H_{\e,\tau}$ is defined in Proposition \ref{prop:klm}. 
\end{Th}
\begin{proof}
Start with a bounded and uniformly Lipschitz continuous function $u :\mathbb{R}^n  \to \R{}$. By Theorem \ref{approx3}, the following inequality holds if $N\tau\geqslant T_0$ for some chosen $T_0 > 1$:  
\begin{equation}\label{vitesse}
\forall\, x \in G_{\e},\quad  |(T^{N\tau}_0 u)(x) - \Tc^N  (u|_{G_{\e}})(x)| \leqslant MN\tau\big(\frac{\e}{\tau}+\tau\big),
\end{equation}
where  $\Tc^N = \Tc_{{0,\e}}^{N\tau}$.   
Dividing by $N\tau$ and letting $N$ go to $\infty$ yields that
$$\left| \overline{ \mathcal H}_{\e,\tau} - \overline H  \right| \leqslant M\big(\frac{\e}{\tau}+\tau\big).$$
\end{proof}

\begin{rem}\rm
\label{rem38}
One may wonder what is the behavior of the quantity $(T^{N\tau}_0 u)|_{G_{\e}}- \Tc^N (u|_{G_{\e}})$ as $T=N\tau\to \infty$. The previous results show that it has a linear growth, of rate $  \overline H - \overline {\mathcal H}_{\e,\tau}  $. Comparing with weak--KAM solutions yields that the second error term is always bounded. However, in some cases more can be said. Indeed, in the autonomous case ($L$ independent of $t$) Fathi proved the convergence of the Lax--Oleinik semi--group (\cite{faconver}), that is, for any initial condition $u$ there exists a weak--KAM solution $u^*$ such that
$(T^{N\tau}_0 u)-N\tau \overline H\to u^*$ uniformly.  Moreover,  it can be proved that the iterated powers of a $(\min,+)$ matrix are periodic after a finite time. Therefore, for $N$ large enough, the sequence $\mathcal \Tc^N (u|_{G_{\e}})$, is periodic after a certain time. In conclusion, in the autonomous case, one can write
$$(T^{N\tau}_0 u)|_{G_{\e}}-  \Tc^N (u|_{G_{\e}})=N\tau( \overline H - \overline{\mathcal  H}_{\e,\tau}) +w_N,$$
where $w_N$ is asymptotic to a periodic sequence.
\end{rem}

\section{Fast (min,plus)--convolution\label{S4}}

As we have seen in \eqref{eq:eqconvol}, the numerical scheme
considered in this paper involves the computation of the
(min,plus)--convolution
$$
 \inf_{y\in G_{\e}} \Big( u(y)+ \tau K^*\big(\frac{x-y}{\tau}\big) \Big), \quad x \in G_{\e}.
$$
In dimension 1, 
if the grid $G_{\e}$ is discretized by retaining $k$  points only, the numerical cost is {\em a priori} of order $k^2$. As we will see now,  we can use a fast (min,plus)--convolution algorithm that turns out to
have a linear cost  (i.e. proportional to $k$) in many situations.

In order to ease the
presentation, we will not deal with functions defined on a grid, but
on functions defined on finite and closed intervals.
 
Let $a,b\in \R{}$ with $a<b$. We write $f:[a,b] \rightarrow \R{}$ if $f$ is such that 
$$ 
\left\{ \begin{array}{ll}
    f(x)<\infty & \text{if } x\in[a,b]\\
    f(x)=\infty & \text{otherwise.}
  \end{array} \right. $$
  
For $f:[a,b]\rightarrow \R{}$, we say that $f$ is respectively convex,
concave, affine if $f_{|[a,b]}$ is respectively convex, concave,
affine.  That is, the functions we consider are defined on
$\R{}$, finite on a closed interval and are said to inherit the properties they
satisfy on this interval.  Let $f:[a,b]\rightarrow \R{}$ and
$g:[c,d]\rightarrow \R{}$. The (min,plus)--convolution (or convolution
in the remaining of the paper) of $f$ and $g$ is defined for all $
x\in\R{}$, by
\begin{equation}
\label{eq:convolfg}
f\conv g(x) = \inf_{y\in\R{}}f(y) + g(x-y).
\end{equation}
Recall that $f \conv g = g \conv f$. 
As $f$ and $g$ are finite only on an interval, it is easy to see that for all $ x\in [a+c,b+d]$,
$$
f\conv g(x) = \inf_{y\in[a,b]}f(y) + g(x-y),
$$
and for all $ x\notin [a+c,b+d]$, $f\conv g(x) =\infty$. 

We will only consider piecewise affine functions and decompose them
according to their affine components: there exist $a_0 = a < a_i <
\cdots < a_n=b$ such that 
$$f = \min_{i\in\{0,\ldots,n-1\}} f_i,$$ where $f_i:[a_i,a_{i+1}] \rightarrow \R{}$ is an affine function. For $i = 0,\ldots,n-1$, we denote by $f'_i = \big(f(a_{i+1})-f(a_i)\big)/ (a_{i+1}-a_i)$ the {\em slope} of $f_i$ or the slope of $f$ on $[a_i,a_{i+1}]$.

\subsection{Convolution}

The fast algorithm to compute the convolution \eqref{eq:convolfg} is based on a decomposition of $g$ $(= u)$ in piecewise convex and concave functions. As the function $f$ $\big(= K^*(\cdot / \tau)\big)$ considered will always be convex \big(see equation \eqref{eq:conv}\big), we thus see that we are led to compute separately the convolution of convex by convex functions, and concave by convex functions defined on finite intervals. As we will see, each block can be computed at a linear cost. In the end, the global cost of the algorithm thus depends on the number of convex and concave components on $f$, a number which might increase in the time evolution of the numerical solution of the Hamilton--Jacobi equation. We will come back later to this matter, but we emphasize that this procedure can be very easily implemented in parallel, each convolution block being computed independently. 


We start with the following result, the proof of which can be found for example \cite{toolbox}. 
\begin{lm}[convolution of a convex function by an affine function]
\label{lm:seg}
  Let $f:[a,b] \rightarrow \R{}$ be a convex piecewise affine function and $g:[c,d]\rightarrow \R{}$ be an affine function of slope $g'$. Then 
$f\conv g: [a+c,b+d] \rightarrow \R{}$ is a convex piecewise affine function defined by
$$ f\conv g(x) = \left\{ \begin{array}{ll}
    f(x-c) + g(c) & \text{ if  } \quad a+c  \leqslant x  \leqslant \alpha+c,\\
    f(\alpha)+g(x-\alpha) & \text{ if  } \quad  \alpha+c < x  \leqslant \alpha+d,\\
    f(x-d) + g(d) & \text{ if  } \quad \alpha+d < x  \leqslant b+d,
  \end{array} \right. $$
where $\alpha = \min\{a_{i} \text{ in the decomposition of $f$} \;|\; f'_i \geqslant g'\}$.
\end{lm}

\begin{figure}[htbp]
  \centering
  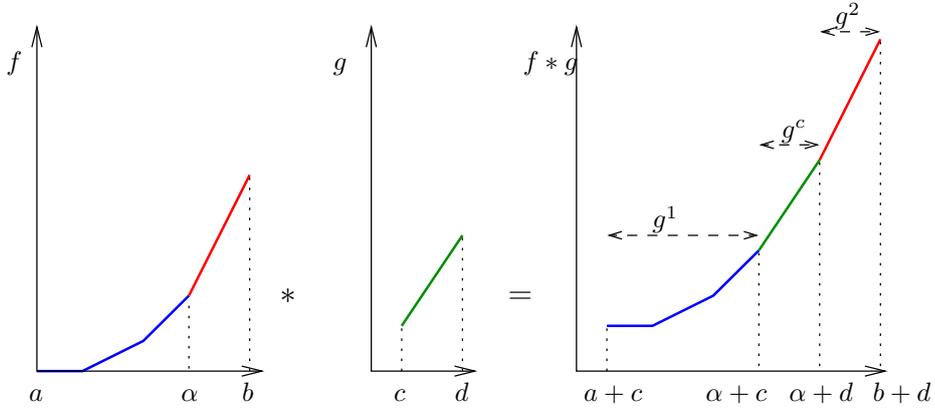
  \caption{Convolution of a convex function by an affine function and decomposition into three functions. }
  \label{fig:conv-seg}
\end{figure}

Figure~\ref{fig:conv-seg} illustrates this lemma. In the rest of the
section, we will use a decomposition of such a convolution into three
parts: $f*g = \min(g^1,g^c,g^2)$, where
\begin{enumerate}
\item[(i)] $g^1 =f*g_{|[c+a,c+\alpha]}$; 
\item[(ii)] $g^c = f*g_{|[c+\alpha,d+\alpha]}$;
\item[(iii)] $g^2 = f*g_{|[d+\alpha,d+b]}$.
\end{enumerate}
 In other words, $g^1$ is composed of
the segments of $f$ whose slope  is   strictly less than that of $g$,
$g^c$ corresponds to the segment $g$ and $g^2$ is composed of the
segments of $f$ whose slope  is   greater than or equal to that of
$g$. Note that $g^c$ is also concave.

A direct consequence of this lemma is Theorem~\ref{th:convex}, stated
in~\cite{Leboudec}. A complete proof is presented in~\cite{BJT2008},
 but can also be deduced from previous works about the Legendre
transform: the Legendre transform is an involution on the set of  convex functions and
can be computed in linear time (see \cite{L97} for example). Moreover,
the transform of the (min,plus)--convolution of two functions is the
addition of the respective transforms of the two functions, inducing an
alternative linear-time algorithm. 
\begin{Th}[convolution of a convex function by a convex function]\label{th:convex}
  If $f$ and $g$ are convex and piecewise affine, then $f\conv g$ is obtained by putting end--to--end the different linear pieces of $f$ and $g$ sorted by increasing slopes.
\end{Th}

For sake of completeness, we give below Algorithm~\ref{algo:convex} for computing the (min,plus)--convolution of two convex piecewise affine functions  without having to compute any Legendre transform. 

%

\begin{algorithm}
\label{algo:convex}
  \KwData{$f:[0,n]\rightarrow \R{}$ a convex function with slopes
    $(r_i)$, $g:[0,m]\rightarrow \R{}$ a convex function with slopes
    $(\rho_i)$.}  
  \KwResult{$h = f\conv g$}
  \Begin{
    $i\gets 0$; $j\gets 0$; $h(0)\gets f(0) + g(0)$\;
    \While{$i+j< n+m$}
    {
      \If{$i\neq n$ and ($r_i < \rho_j$ or $j=m$)}{$h(i+j+1)\gets h(i+j) + r_i$; $i\gets i+1$\;}
      \Else{$h(i+j+1)\gets h(i+j) + \rho_j$; $j\gets j+1$\;}}
}
\caption{Convolution of two convex functions}
\end{algorithm}

\medskip 

We now turn to the case where $f$ is convex and $g$ is concave, 
and the Legendre tranform cannot help anymore to design an
algorithm as in \cite{brenier,L97}.  We begin with the following lemma:

\begin{lm}
  Let $f:[a,b] \rightarrow \R{}$ be a convex piecewise affine function
  and $g:[c,d]\rightarrow \R{}$ be a concave piecewise affine function which decomposition is $g=\min\limits_{j=1}^m g_j$. Then
  $$ f\conv g = \min_{j=1}^m f\conv g_j.$$
\end{lm}

\begin{proof}
This is a direct consequence of the distributivity of $\conv$ over the minimum.
\end{proof}

Now, consider two functions $f:[a,b]\rightarrow \R{}$, convex, and
$g:[c,d]\rightarrow \R{}$, concave, with respective decompositions in
$f_i:[a_i,a_{i+1}]\rightarrow \R{}$, $i\in\{0,n-1\}$ and
$g_j:[c_j,c_{j+1}]\rightarrow \R{}$, $j\in\{0,m-1\}$. The following
lemma, that considers two consecutive affine functions of $g$, leads
to an efficient algorithm to compute the convolution of a convex
function by a concave function.

\begin{lm}
\label{lm:simplif}
Consider the convolutions $f\conv g_{j-1}$ and $f\conv g_{j}$. Let 
$$
\alpha_j
= \min\{a_{i} \text{ in  the decomposition of $f$} \;|\; f'_i\geqslant
g'_j\}
$$   
and 
$$
\alpha_{j-1}
= \min\{a_{i} \text{ in the decomposition of $f$} \;|\; f'_i\geqslant
g'_{j-1}\}.
$$
 Then
  \begin{itemize}
  \item $\forall x \leqslant c_{j} + \alpha_{j}$,\ \  $f\conv g_{j}(x)  \geqslant f\conv g_{j-1}(x)$;
  \item $\forall x \geqslant c_{j} + \alpha_{j-1}$,\ \  $f\conv g_{j-1}(x)  \geqslant f\conv g_{j}(x)$.
  \end{itemize}
\end{lm}

\begin{proof}
  First, as $f$ is convex and $g'_{j-1} > g'_{j}$, we have that $\alpha_{j-1} \geqslant
  \alpha_{j}$.  
  From Lemma~\ref{lm:seg}, for all $x  \leqslant c_{j} + \alpha_{j}$,
  $$f\conv g_{j}(x) = f(x-c_{j}) + g(c_{j}).$$ 

  Either $x  \leqslant c_{j-1} + \alpha_{j-1}$, then $f\conv g_{j-1}(x) = f(x-c_{j-1}) +
  g(c_{j-1})$ and 
  $$
  f\conv g_{j}(x) - f\conv g_{j-1}(x) = f(x-c_{j})-
  f(x-c_{j-1}) + g(c_{j}) - g(c_{j-1});
  $$ 
  as $x-c_{j-1}  \leqslant \alpha_{j-1}$, then
  $f(x-c_{j-1})-f(x-c_{j})  \leqslant g'_{j-1}\!\cdot\!(c_j-c_{j-1})$ and $f\conv
  g_{j}(x) - f\conv g_{j-1}(x) \geqslant 0$;

  or $x > c_{j-1} + \alpha_{j-1}$, then $f\conv g_{j-1}(x) = f(\alpha_{j-1}) +
  g(x-\alpha_{j-1})$; as $c_{j-1}<x-\alpha_{j-1} \leqslant x-\alpha_j \leqslant c_j$, $g(c_{j})-g(x-\alpha_{j-1}) =
  g'_{j-1}\!\cdot\! (c_j+\alpha_{j-1}-x)$ and $f(\alpha_{j-1}) - f(x-c_j) \leqslant
  g'_{j-1}\!\cdot\!(c_j+\alpha_{j-1}-x)$; then $f\conv g_{j}(x) - f\conv g_{j-1}(x)\geqslant
  0$.

  The second statement can be proved similarly. 
\end{proof}

\begin{figure}[htbp]
  \centering
  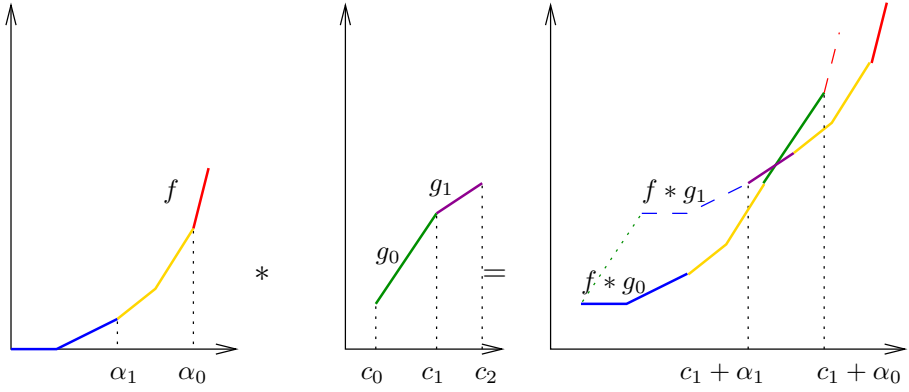
  \caption{Convolution of a convex function by a concave function. }
  \label{fig:conv-sconc}
\end{figure}

Another formulation of Lemma~\ref{lm:simplif} is that $g^1_j\geqslant f*g_{j-1}$ and that $g^2_{j-1}\geqslant f*g_j$ and that the two functions intersect at least once. Hence $g_j^1$ and $g^2_{j-1}$ cannot appear in the minimum of $f*g_j$ and $f*g_{j-1}$. By transitivity, there is no need to
compute entirely the convolution of the convex function by every
affine component of the decomposition of the concave function. If
there are more than two segments, successive applications of this lemma
show that only the position of the segments of the concave function must be
computed, except for the extremal segments. 

The following lemma shows that $f*g_{j}$ and $f*g_{j-1}$ intersect  in one
and only one connected component,  as for a given abscissa, the slope of $f*g_j$ is less
than the one of $f*g_{j-1}$.
\begin{lm}
  Let $x\in\R$ be any real number for which both $f*g_j$ and $f*g_{j-1}$ are finite valued and differentiable. Then
  \begin{equation}\label{comparaison}
  \frac{\d}{\d x} f*g_j(x) \leqslant \frac{\d}{\d x}f*g_{j-1}(x).
  \end{equation}
\end{lm}

\begin{proof}
\label{lm:deriv}

  For $x\in[a_0, \alpha_j]$,
  $f*g_{j}(x+c_j) = f(x-c_j) + g(c_j)$ and $f*g_{j-1}(x+c_{j-1}) = f(x)+
  g(c_{j-1})$. As $f$ is convex and $c_j\geqslant c_{j-1}$, the result holds on $[a_0+c_j, \alpha_j+c_j]$. Similarily, the
  result holds for $x\in[\alpha_{j-1}+c_j, a_n+c_j]$.
  
  On $[c_j+\alpha_j, c_j+\alpha_{j-1}]$, $f*g_{j}$ is composed
  of segment $g_{j}$ concatenated with the segments $f_i$,
  $i\in[\alpha_{j-1},\alpha_j]$, possibly truncated on the right and
  $f*g_{j-1}$ is composed of segments $f_i$,
  $i\in[\alpha_{j-1},\alpha_j]$ concatenated with $g_{j-1}$, possibly
  truncated on the left. As $\forall i \in[\alpha_{j-1},\alpha_j]$,
  $g'_{j-1}\leqslant f'_i \leqslant g'_j$, the result holds.
\end{proof}

If one sets by convention $\frac{\d}{\d x}f*g_j(x) = -\infty$ for
$x<c_{j-1}+a_0$ and $\frac{\d}{\d x}f*g_j(x) = +\infty$ for
$x>c_{j}+a_n$, then the inequality always holds. 

The intersection of $f*g_{j-1}$ and $f*g_{j}$ can then happen in one and only one of the four cases: 
\begin{enumerate}
\begin{minipage}[t]{0.45\linewidth}
\item $g^1_{j-1}$ and $g^c_j$ intersect; 
\item $g^1_{j-1}$ and $g^2_j$ intersect; 
\end{minipage}
\begin{minipage}[t]{0.45\linewidth}
\item $g^c_{j-1}$ and $g^c_j$ intersect; 
\item $g^c_{j-1}$ and $g^2_j$ intersect.
\end{minipage}
\end{enumerate}

The following theorem is another 
consequence of these lemmas and is more precise about the shape on the convolution of a convex function by a concave function.

\begin{Th}[convolution of a convex function by a
  concave function]
\label{th:ccc}  The (min,plus)--convolution of a convex function by a concave function can be decomposed in three (possibly trivial) parts:
a convex function, a concave function and a convex function.
\end{Th}

\begin{proof}
  We use the notations defined in the former
  lemmas. 

  We now show by induction that the convolution of $f$ by $\min\limits_{\lambda\leqslant j}g_j$,
  denoted $h_j$, is composed of 
  \begin{enumerate}
  \item[(i)] a convex part $h_j^1$, which is the a
  restriction of $g_0(x)=f(x-c_0)+g(c_0)$ to $[c_0+a_0,\beta_j]$, with $\beta_j\leqslant
  c_0+a_n$; 
  \item[(ii)] a concave part $h_j^c$, which is a minimum of some segments
  $g_{\lambda}$, $\lambda\leqslant j$ (up to some translation) finite valued on $[\beta_j,\gamma_j]$; 
  \item[(iii)] a convex part $h_j^2$,
  $g_j(x)=f(x-c_{j+1})+g(c_{j+1})$ for $x\in[\gamma_j,a_n+c_{j+1}]$, $\gamma_j\geqslant a_0+c_{j+1}$.
  \end{enumerate}


Note that with these conventions, the real numbers $\beta_j$ and $\gamma_j$ are uniquely determined at each step of the induction.

  The case with $j=0$ is a direct consequence of Lemma~\ref{lm:seg}.
  The case with $j=1$ is a consequence of Lemmas~\ref{lm:simplif} and~\ref{lm:deriv}.
  The graphs of $f*g_1$ and $f*g_0$ intersect once and only once (where they are
  finite valued), and in $[c_1+\alpha_1, c_1+\alpha_0]$. Depending on when
  this intersection occurs, the concave part will be trivial, be made
  of only one (part of a) segment of $g$, or a minimum of the two segments $g_0$
  and $g_1$. 
  
  Suppose now that the result holds for $h_j$ and consider $h_{j+1} =
  \min (h_j, f*g_{j+1})$. The argument is exactly the same as for
  $j=1$: $h_j$ and $f*g_{j+1}$ can only intersect once and only once. Indeed, $h_j$ is the
  minimum of functions such that $\frac{\d}{\d x}f*g_k (x)\geqslant
  \frac{\d}{\d x}f*g_{j+1}(x)$, and then
  $\frac{\d}{\d x}h_j(x)\geqslant\frac{\d}{\d x}f*g_{j+1}(x)$. Note that this intersection has to occur after  the point $c_{j+1}+\alpha_{j+1}$. 
  
  Moreover, as $g_j^2$ does not intersect $g_{j+1}^2$  and that $h_j^2$ is a part of $g_j^2$ (by the induction hypothesis), $h_j^2$ does not intersect $g_{j+1}^2$.

  Therefore, only one of the four following cases may occur.

  \begin{enumerate}
  \item $h_j^c$ intersects $g_{j+1}^c$ and $h_{j+1}^1 = h_j^1$,
    $h_{j+1}^c = \min(h_j^c,g_{j+1}^c)$, $h_{j+1}^2 = g_{j+1}^2$,
    $\beta_{j+1} = \beta_j$ and $\gamma_{j+1} = \alpha_{j+1}+c_{j+2}$.
  \item $h_j^c$ intersects $g_{j+1}^2$ at $y$ and $h_{j+1}^1 = h_j^1$, $h_{j+1}^c = h_{j}^c$, $h_{j+1}^2 = g_{j+1}^2$,
    $\beta_{j+1} = \beta_j$ and $\gamma_{j+1} = y$.
  \item $h_j^1$ intersects $g_{j+1}^c$ at $y$ and $h_{j+1}^1 = h_j^1$, $h_{j+1}^c = h_{j}^c$, $h_{j+1}^2 = g_{j+1}^2$,
    $\beta_{j+1} = y$ and $\gamma_{j+1} = \alpha_{j+1}+c_{j+2}$.
  \item $h_j^1$ intersects $g_{j+1}^2$ at $y$ and $h_{j+1}^1 = h_j^1$, $h_{j+1}^c$ is trivial and $h_{j+1}^2 = g_{j+1}^2$,
    $\beta_{j+1} = \gamma_{j+1} = y$.
  \end{enumerate}
\end{proof}

If the concave function is composed of $m$ segments and the convex
function of $n$ segments, then the convolution of those two functions
can be computed in time $\cO(n+m\log m)$. The $\log m$ term comes from
the fact that one has to compute the minimum of $m$ segments (see
\cite{toolbox} for more details). If the functions are now defined on $\N$, 
then, as no intersection point has to be computed for the minimum, the time complexity is $O(n+m)$. The corresponding algorithm is given in Algorithm~\ref{algo:concave}, where 
without loss of generality (the (min,plus)--convolution is
shift-invariant), the functions $f$ and $g$ are defined on $\N$ and
finite between respectively $0$ and $n$, and $0$ and $m$. The slopes of the functions are thus $f'_i = f(i)-f({i-1})$ and $g'_i =
g(i)-g({i-1})$. 
\begin{algorithm}
  \label{algo:concave}
  \KwData{$f:[0,n]\rightarrow \R{}$ a convex function with slopes
    $(r_i)$, $g:[0,m]\rightarrow \R{}$ a concave function with slopes
    $(\rho_i)$.}  
  \KwResult{$h = f\conv g$}
  \Begin{
    \tcc{Initialization}
    $k\gets 0$\;
    \lWhile{$k \leqslant m+n$}{$h(k)\gets +\infty$; $k\gets k+1$\;}
    $i\gets 0$; $j\gets 0$; $h(0)\gets f(0) + g(0)$\;
    \tcc{First convex part of the convolution}
    \While{$f'_i  \leqslant g'_0$}{$i\gets i+1$; $h(i) \gets f(i) + g(0)$\;}
    \tcc{Concave part of the convolution}
    $j\gets j+1$; $h(i+j) \gets f(i) + g(j)$\;
    \While{$j < m$}{
      \lWhile{$g'_j< f'_{i-1}$}{$i\gets i-1$\;}
      $h(i+j)\gets \min (h(i+j),f(i)+g(j))$\;
      $h(i+j+1) \gets \min (h(i+j+1),f(i)+g(j+1))$\;
      $j\gets j+1 $\; }
    \tcc{Second convex part of the convolution}
    \While{$i<n$}{$i\gets i+1$; $h(i+m)\gets \min (h(i+m),f(i)+g(m))$\;}
  }
\caption{Convolution of a convex function by a concave function}
\end{algorithm}

\subsection{Application}
We now go back to our initial problem. As already explained above, the computation of $\Tc_{t,\e}^\tau u (t,x)$ given in \eqref{eq:eqconvol} is made of two steps:
\begin{itemize}
\item (min,plus)--convolution of $u$ and $h:x\mapsto \tau K^*(\frac{x}{\tau})$;
\item subtract $\tau V(t,x)$. 
\end{itemize}

Note that the (min,plus)--convolution described above is here defined on functions that
have an unbounded support. But in the periodic case, $u$ is 1--periodic and
$h$ is convex with a global minimum. Then, to compute the convolution,
it is enough to compute it on a single period (the
(min,plus--convolution preserves the periodicity), and replace $h$ by its 
restriction on a support of size $2$ centered on its  minimum. If $\e = 1/k$ with the notation of the previous section, then both functions $u$ and $h$ are defined on grids of size $k$ and $2k$ respectively. 

The convolution of $u$ and $h$ can be efficiently computed following these steps:

\begin{enumerate}
\item Decompose $u$ into convex and concave parts. This can be done in linear time: the three first points determine if a part is concave or convex. Then, this part is extended as much as possible while preserving the concavity or convexity and so on.
\item For each convex or concave part, perform the convolution with $h$ using Algorithms~\ref{algo:convex} or~\ref{algo:concave}.
\item Take the minimum of all these convolutions. 
\end{enumerate}
The complexity of this Algorithm is then $\mathcal{O}(ck)$, where $c$
is the number of components in the decomposition of $u$ into concave/convex parts.
%

\subsection{Implementation issues}

The main issue with this algorithm is that $c$ -- the number of components in the decomposition of $u$ -- can become very large,
and then lead to a quadratic time complexity, which is the complexity
of a  naive algorithm for computing the convolution.  Experimentally, the
reason for this is that, due to the discretization of $u$, nearly
affine parts, after performing the convolution several times, are
computed as fast alternations of convex and concave parts. As shown in
Figure~\ref{fig:toler}, one solution to make the computations more
efficient would be to consider those parts are convex and use
Algorithm~\ref{algo:convex}.

\begin{figure}[htbp]
  \centering
  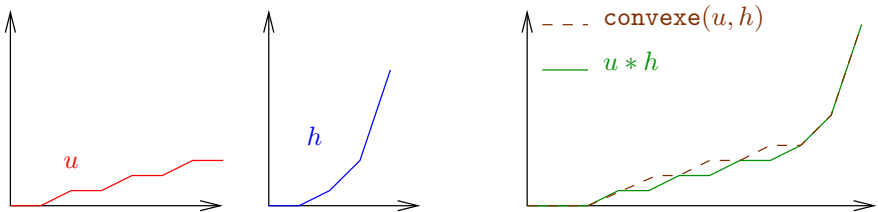
  \caption{Approximation of the convolution: plain line shows the convolution of $u$ and $h$, and the dashed line shows the function computed using Algorithm~\ref{algo:convex} when $u$ is not convex, but has very small variations. }
  \label{fig:toler}
\end{figure}

To do this, we decompose $u$ into convex
and concave parts with a tolerance (we do not request for convex parts
to have increasing increments, but the increments to have an increase
more than $-\eta$). We will discuss this in the next section. The choice of an optimal tolerance $\eta$, as well as the comparison with parallel implementations, will be the subject of further studies. 

\section{Numerical simulations}

\subsection{Pendulum}

We take the Hamiltonian \eqref{eq:HPV} with $P = 0$ and $V(t,x) = 1 - \cos(x)$ on $(0,2\pi]\times \R$,  with periodic boundary conditions. The initial value is $u_0(x) = \cos(x)$.  
In this case the corresponding solution develops a singularity in the derivative in $x = \pm \pi$ and the solution is not smooth. The numerical solution at $T = 8$ is plot in figure \ref{figpendulum}.

\begin{figure}[ht]
\begin{center}
\rotatebox{0}{\resizebox{!}{0.40\linewidth}{%
   \includegraphics{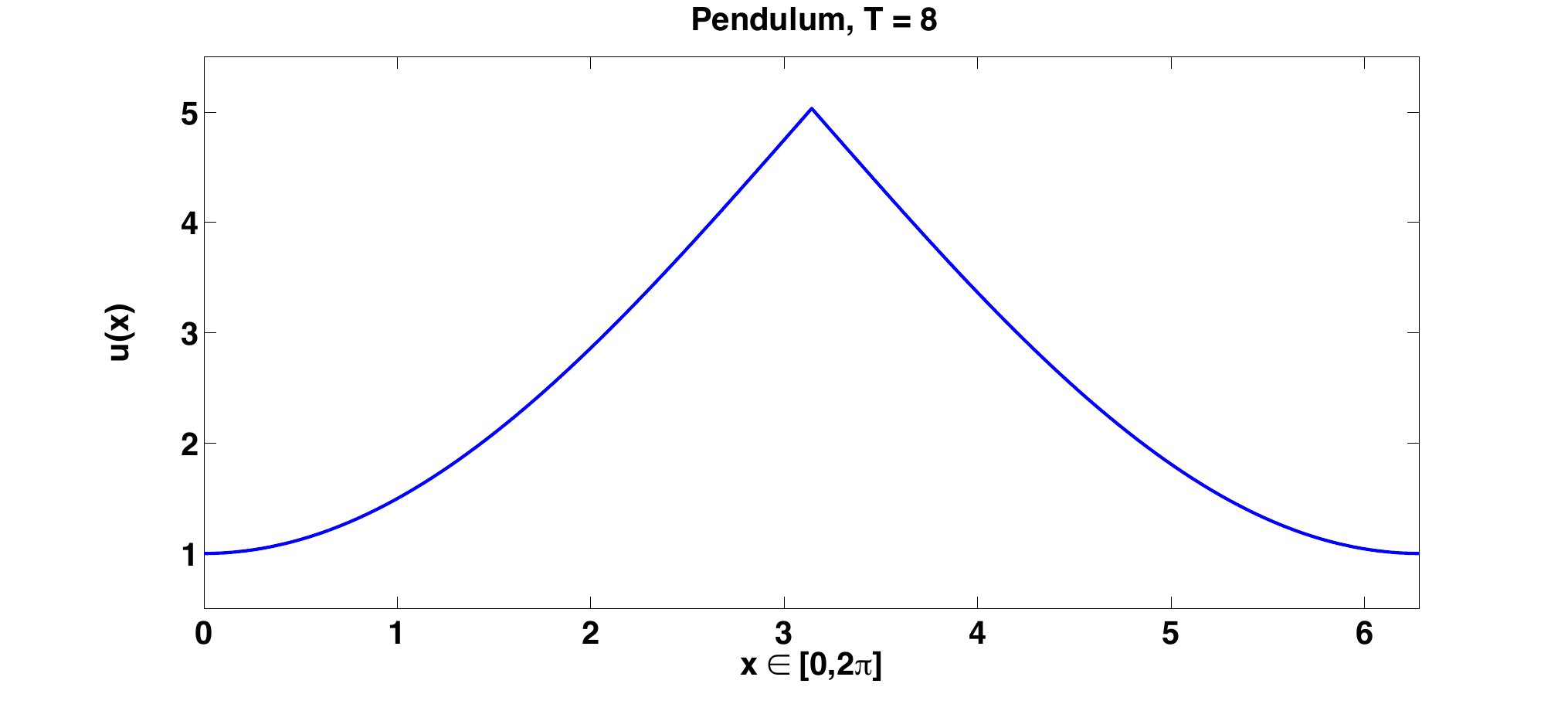}}}
  \end{center}
\caption{Solution of the Hamilton-Jacobi equation at $T = 8$ for the potential $1 - \cos(x)$ (pendulum)}
\label{figpendulum}
\end{figure}

We perform different simulations with mesh size of the form $\pi/K$ with $K = 2^{n}$ for $n = 1$ to $n = 13$ corresponding to $2K$ grid points between $0$ and $\pi$ (from $4$ to $16384$ grid points). The time discretization parameter $\tau$ is taken to be $\sqrt{h}$ so that we expect a global order of convergence $\mathcal{O}(\sqrt{h})$. 

The result is plot on figure \ref{figerr1} where the error is the relative error in $L^\infty$ norm between the solution obtained with the algorithm described above at time $T = 8$ and the solution computed using the fifth-order WENO algorithm   (WENO5, see \cite{jiangshu}) with $8192$ grid points. Note that in this first simulation we  take the regularisation parameter $\eta = 0$. We observe the expected convergence rate.

\begin{figure}[ht]
\begin{center}
\rotatebox{0}{\resizebox{!}{0.40\linewidth}{%
   \includegraphics{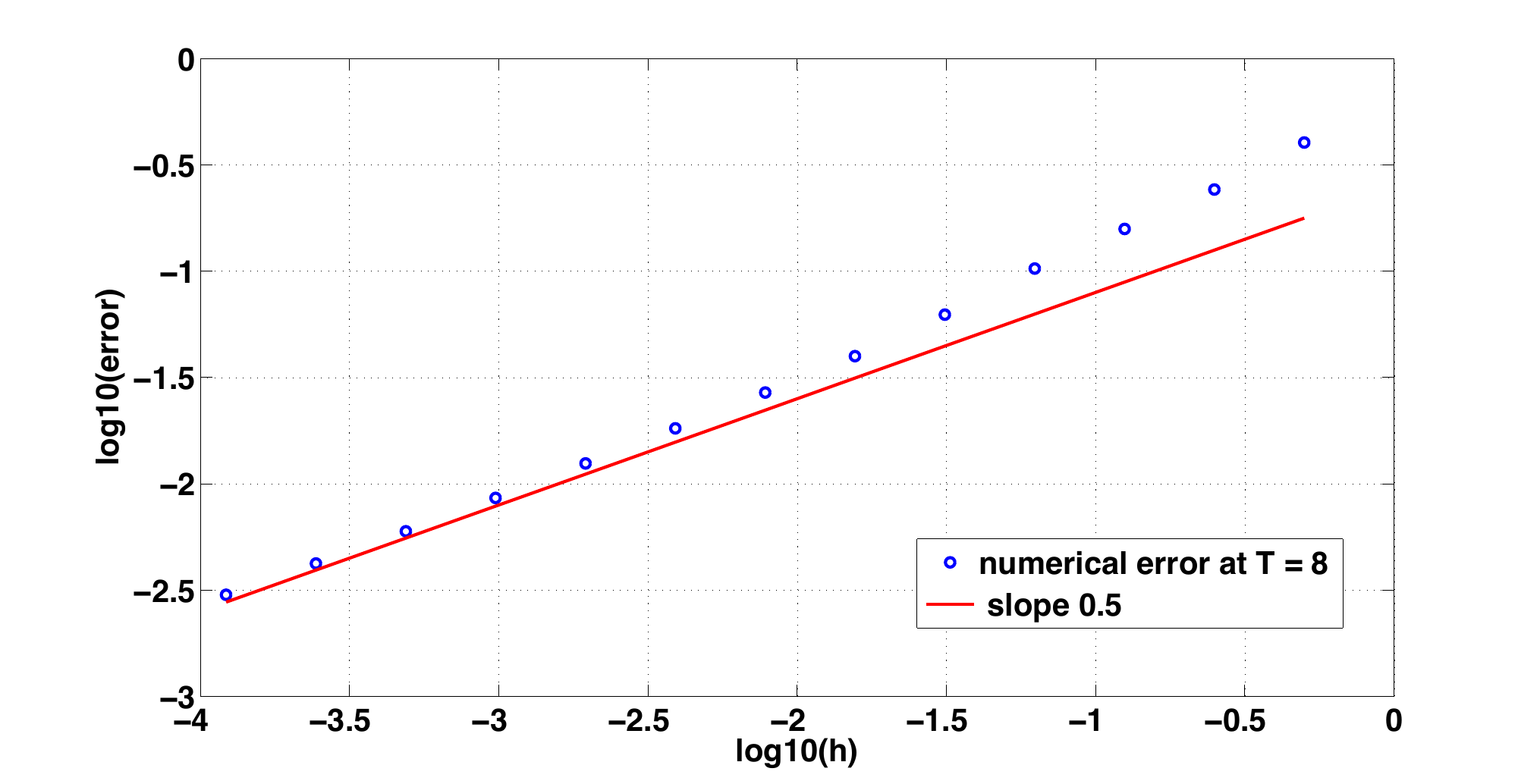}}}
  \end{center}
\caption{Error versus mesh size $ h$ with $\tau = \sqrt{h}$. Regularization parameter $\eta = 0$ (pendulum)}
\label{figerr1}
\end{figure}

In a second step we perform the same simulation, but with $\eta = h$, which does not affect the convergence rate, but allows to go up to  $2K = 2^{18} = 262144$ grid points for a few minutes of CPU time\footnote{The CPU time required to obtain the solution with $2^{18}$ grid points is about 19mn, using MATLAB on a Mac power book 2,3 GHz Intel Core i7}

\begin{figure}[ht]
\begin{center}
\rotatebox{0}{\resizebox{!}{0.40\linewidth}{%
   \includegraphics{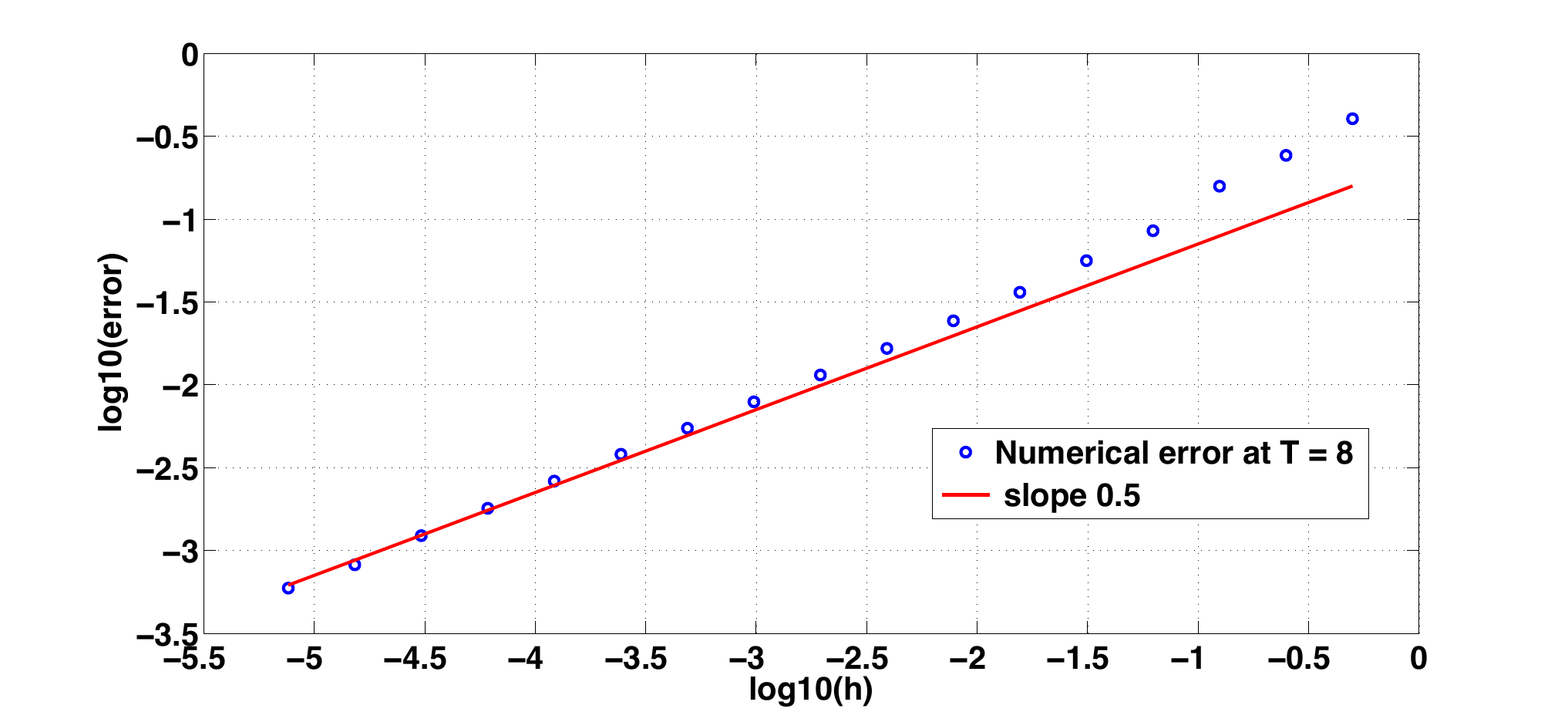}}}
  \end{center}
\caption{Error versus mesh size $ h$ with $\tau = \sqrt{h}$. Regularization parameter $\eta = h$ (pendulum) }
\label{figerr2}
\end{figure}

Using the same reference solution at $T = 8$, we perform several simulations with different values of the regularization parameter. In Figure \ref{etaerror}, we plot the evolution of the error with respect to $\eta$ and for different values of $h$. 

\begin{figure}[ht]
\begin{center}
\rotatebox{0}{\resizebox{!}{0.40\linewidth}{%
   \includegraphics{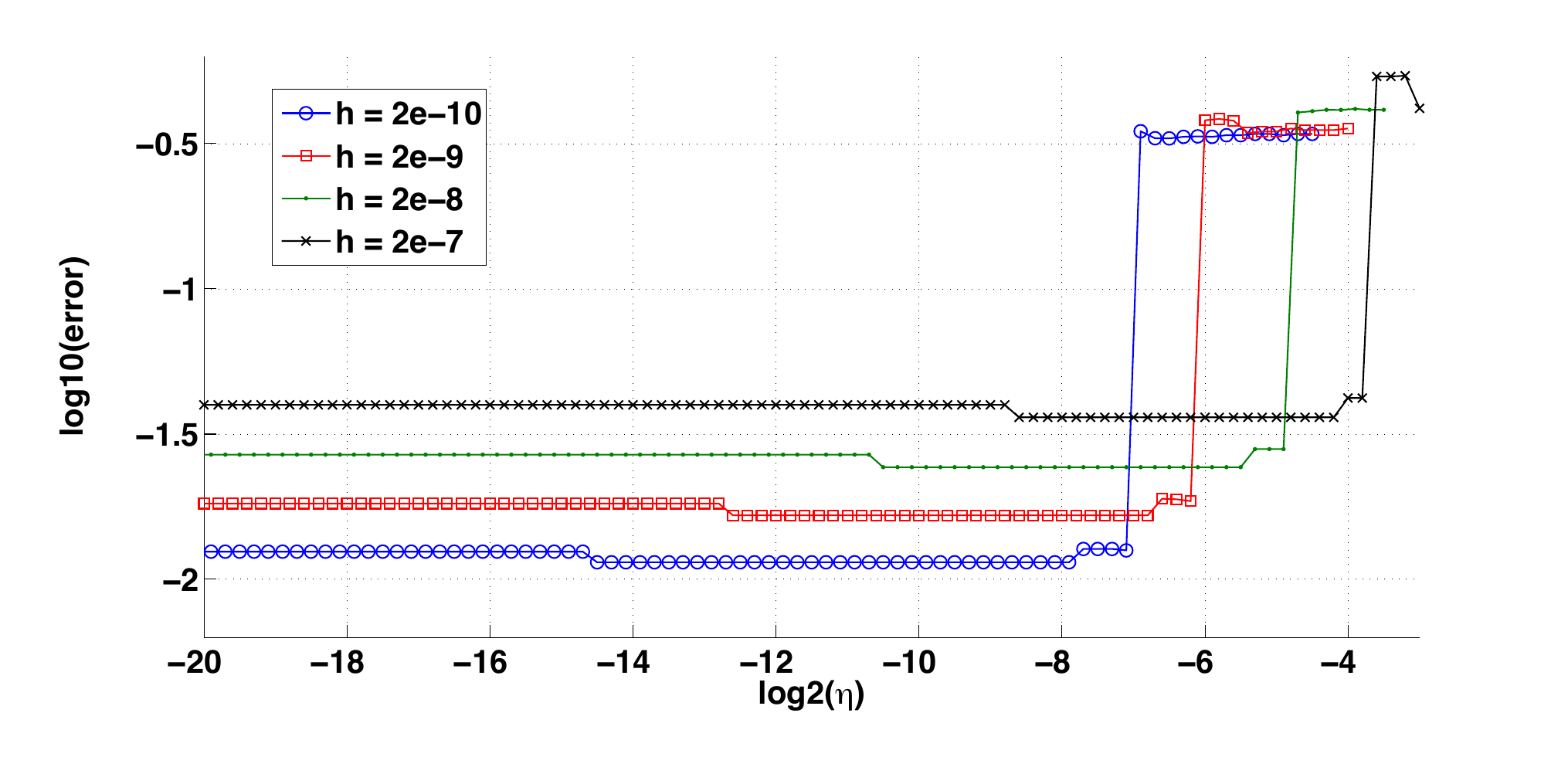}}}
  \end{center}
\caption{Error versus regularization parameter $\eta$ for different values of $h$ (pendulum) }
\label{etaerror}
\end{figure}

In Figure \ref{etaCPU} we plot the evolution of the CPU time with respect to $\eta$, and in Figure \ref{CPUerror}, we plot the error versus the CPU time, for different values of $h$ and different values of $\eta$.

\begin{figure}[ht]
\begin{center}
\rotatebox{0}{\resizebox{!}{0.40\linewidth}{%
   \includegraphics{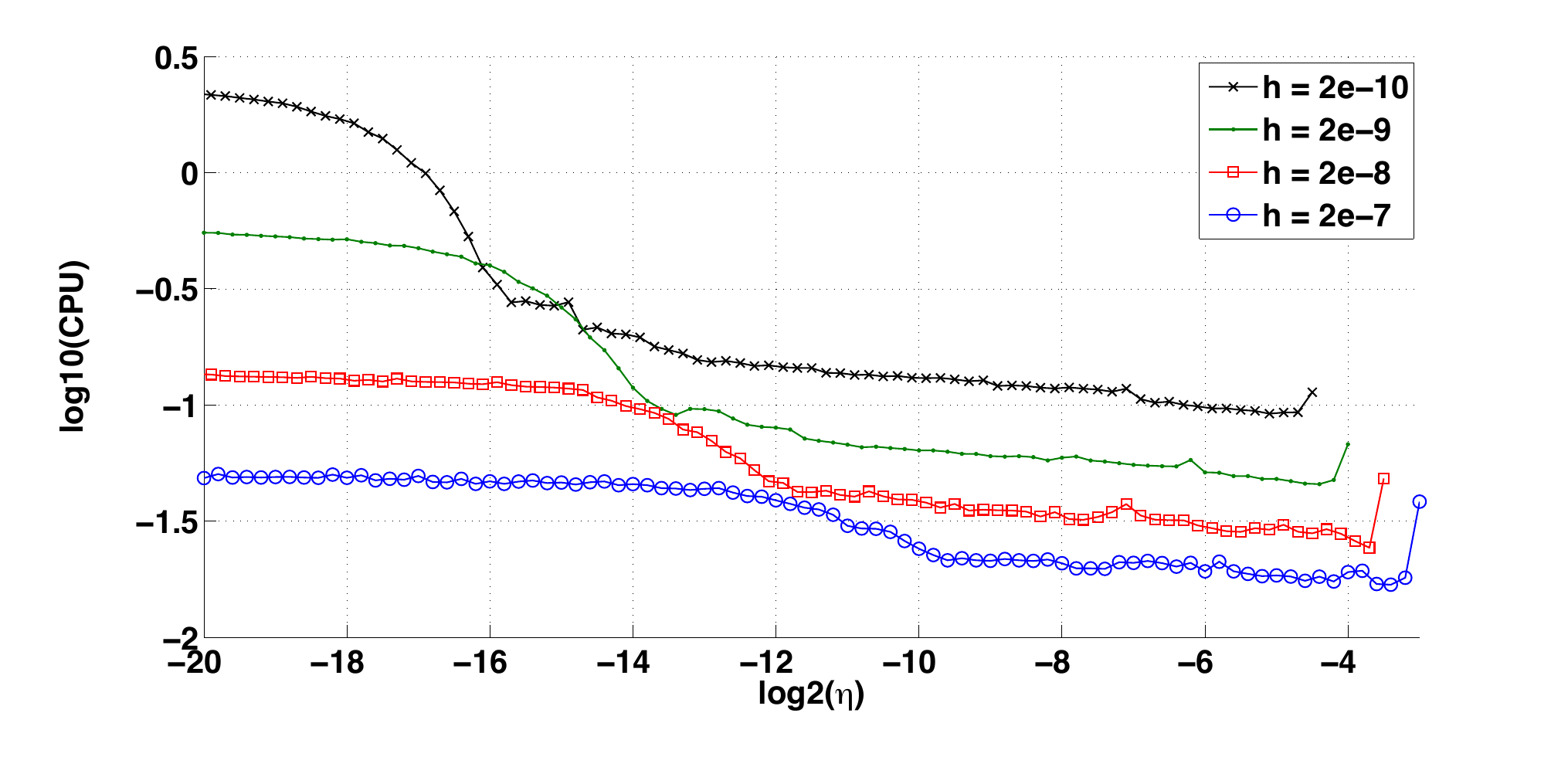}}}
  \end{center}
\caption{CPU time versus regularization parameter $\eta$ for different values of $h$ (pendulum). }
\label{etaCPU}
\end{figure}

\begin{figure}[ht]
\begin{center}
\rotatebox{0}{\resizebox{!}{0.40\linewidth}{%
   \includegraphics{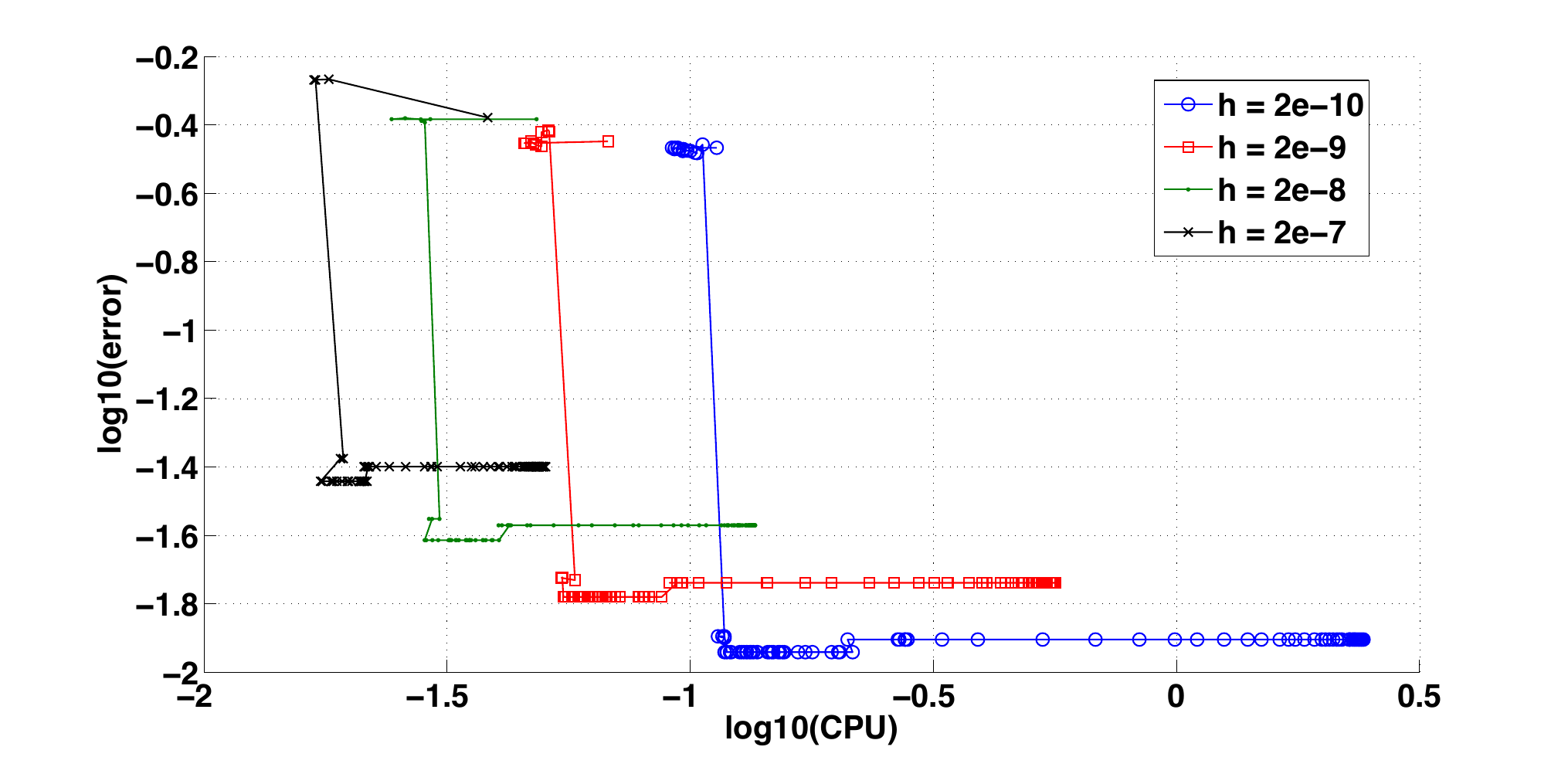}}}
  \end{center}
\caption{Error versus CPU time for different values of $h$ (pendulum). }
\label{CPUerror}
\end{figure}

As a conclusion of these simulations, it seems that the performances of the algorithm seem to be optimized for $\eta$ depending linearly of $h$ in the case of the pendulum. 

\subsection{Time dependent potential}

We now take the $P = 1$ and the time dependent potential $V(t,x) = \cos(2x) \sin(t)$. We still consider the initial eigenvalue $u_0(x) = \cos(x)$. In this case,  a theorem of Bernard and Roquejoffre (\cite{BJ})  states that the solution converges towards a function that is periodic in time (of period possibly greater than $2\pi$) which is a priori not constant, in contrast with the previous case. The shape of the solution is depicted in Figure \ref{figpendt}. 

\begin{figure}[ht]
\begin{center}
\rotatebox{0}{\resizebox{!}{0.40\linewidth}{%
   \includegraphics{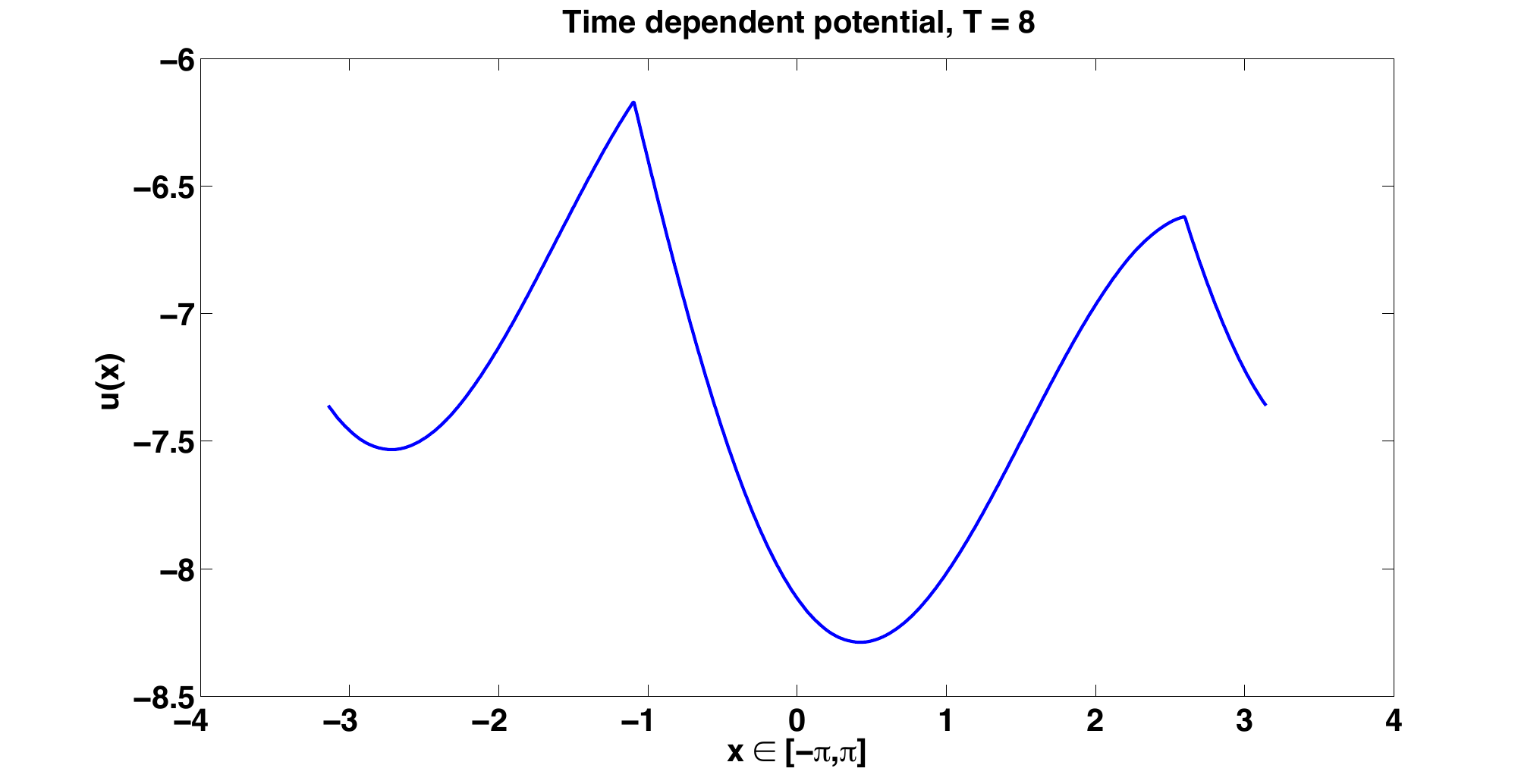}}}
  \end{center}
\caption{Solution of the Hamilton-Jacobi equation at $T = 8$ for the potential $\cos(2x)\sin(t)$ (time dependent potential)}
\label{figpendt}
\end{figure}

In Figures \ref{figerr1t} we illustrate the convergence result obtained above in the case $\tau = \sqrt{h}$ and observe the predicted rate of convergence $h^{1/2}$. Again the exact solution is computed at $T = 8$ with the WENO5 algorithm. 

\begin{figure}[ht]
\begin{center}
\rotatebox{0}{\resizebox{!}{0.40\linewidth}{%
   \includegraphics{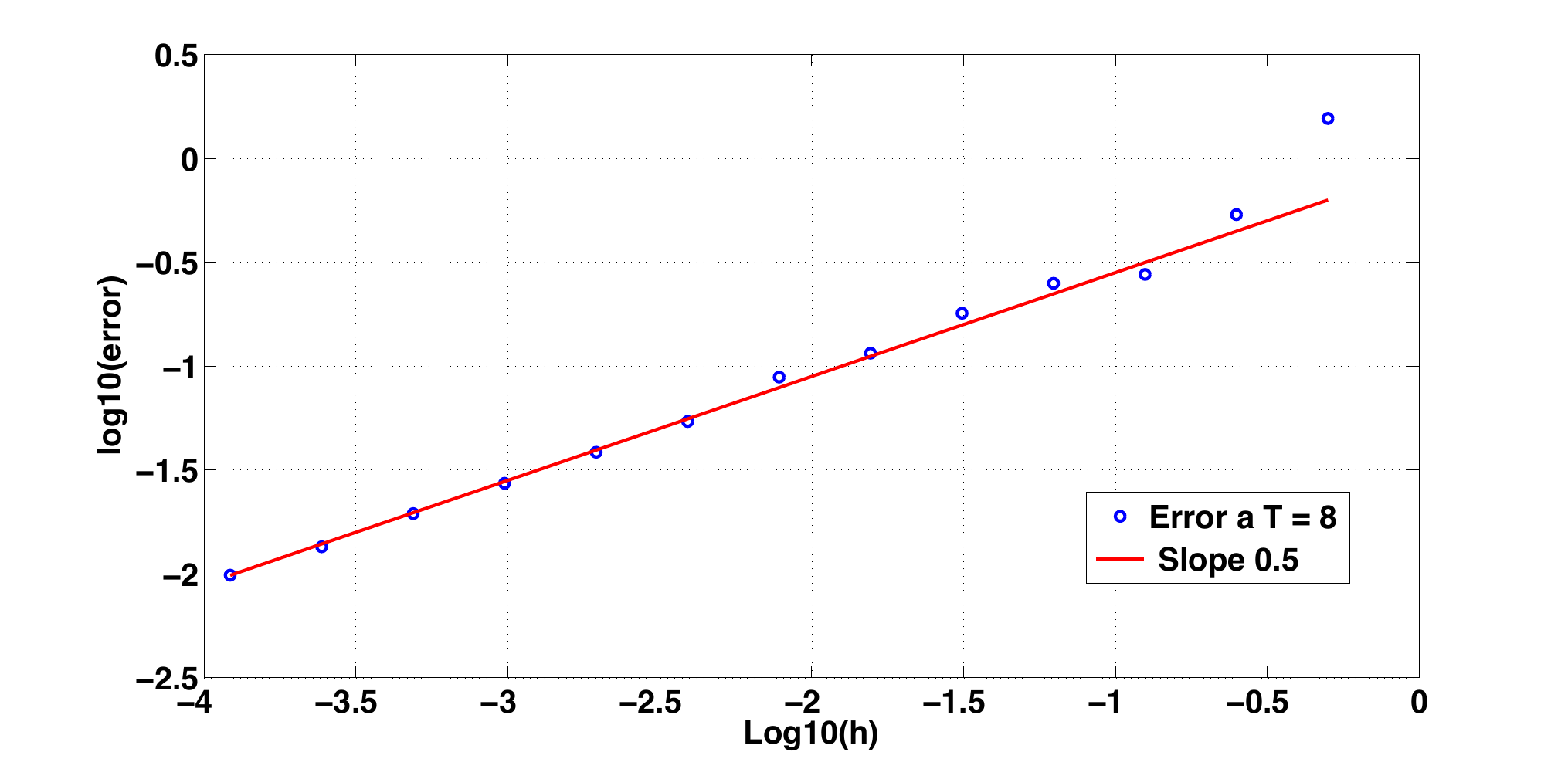}}}
  \end{center}
\caption{Error versus mesh size $ h$ with $\tau = \sqrt{h}$. Regularization parameter $\eta = 0$ (time dependent potential)}
\label{figerr1t}
\end{figure}

In figures \ref{etaerrort}, \ref{etaCPUt} and \ref{CPUerrort} we study the effect of the regularization parameter $\eta$ with the same data as in the previous case. We see that the same conclusion can be drawn. 
\begin{figure}[ht]
\begin{center}
\rotatebox{0}{\resizebox{!}{0.40\linewidth}{%
   \includegraphics{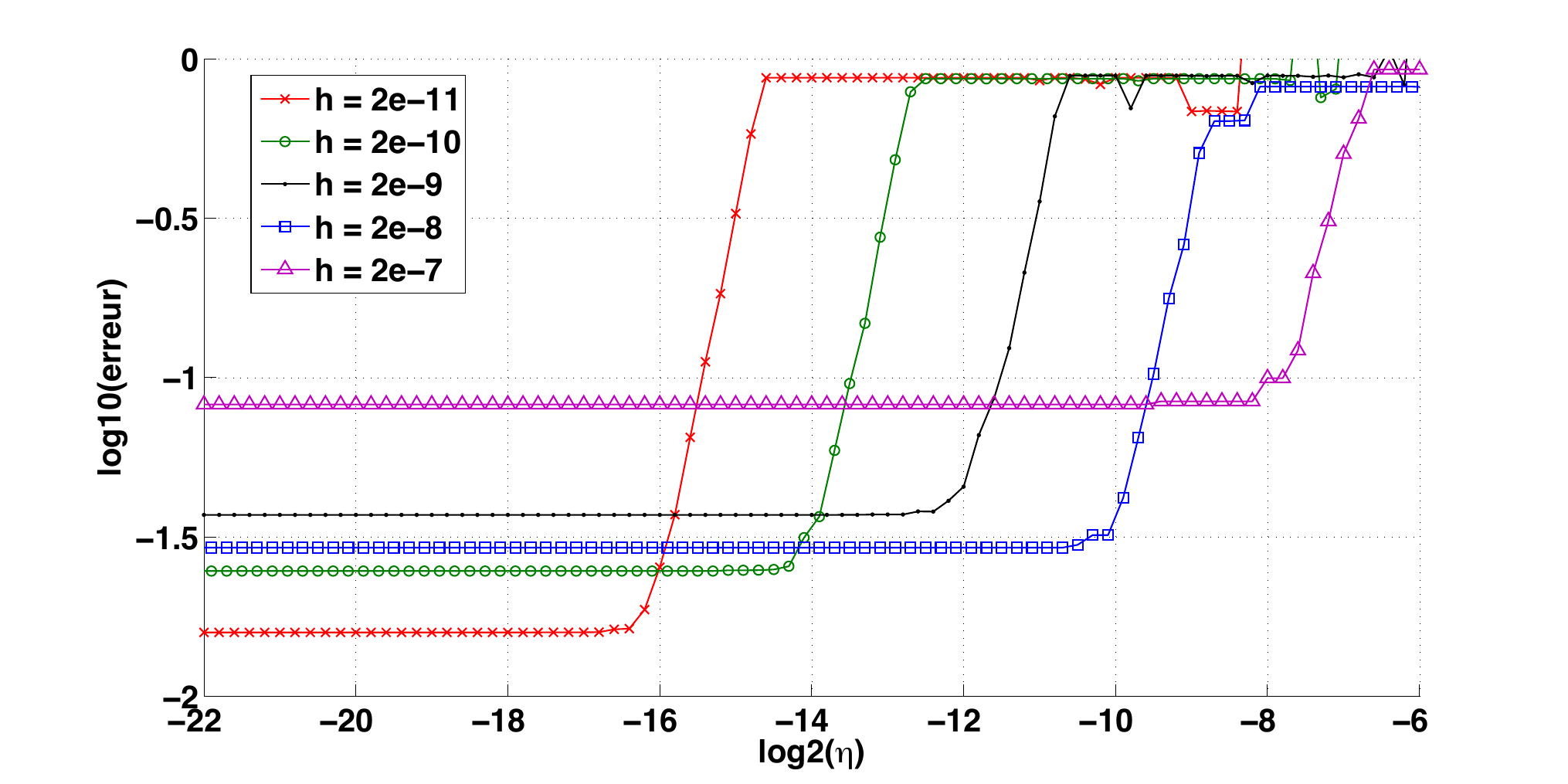}}}
  \end{center}
\caption{Error versus regularization parameter $\eta$ for different values of $h$ (time dependent potential) }
\label{etaerrort}
\end{figure}

\begin{figure}[ht]
\begin{center}
\rotatebox{0}{\resizebox{!}{0.40\linewidth}{%
   \includegraphics{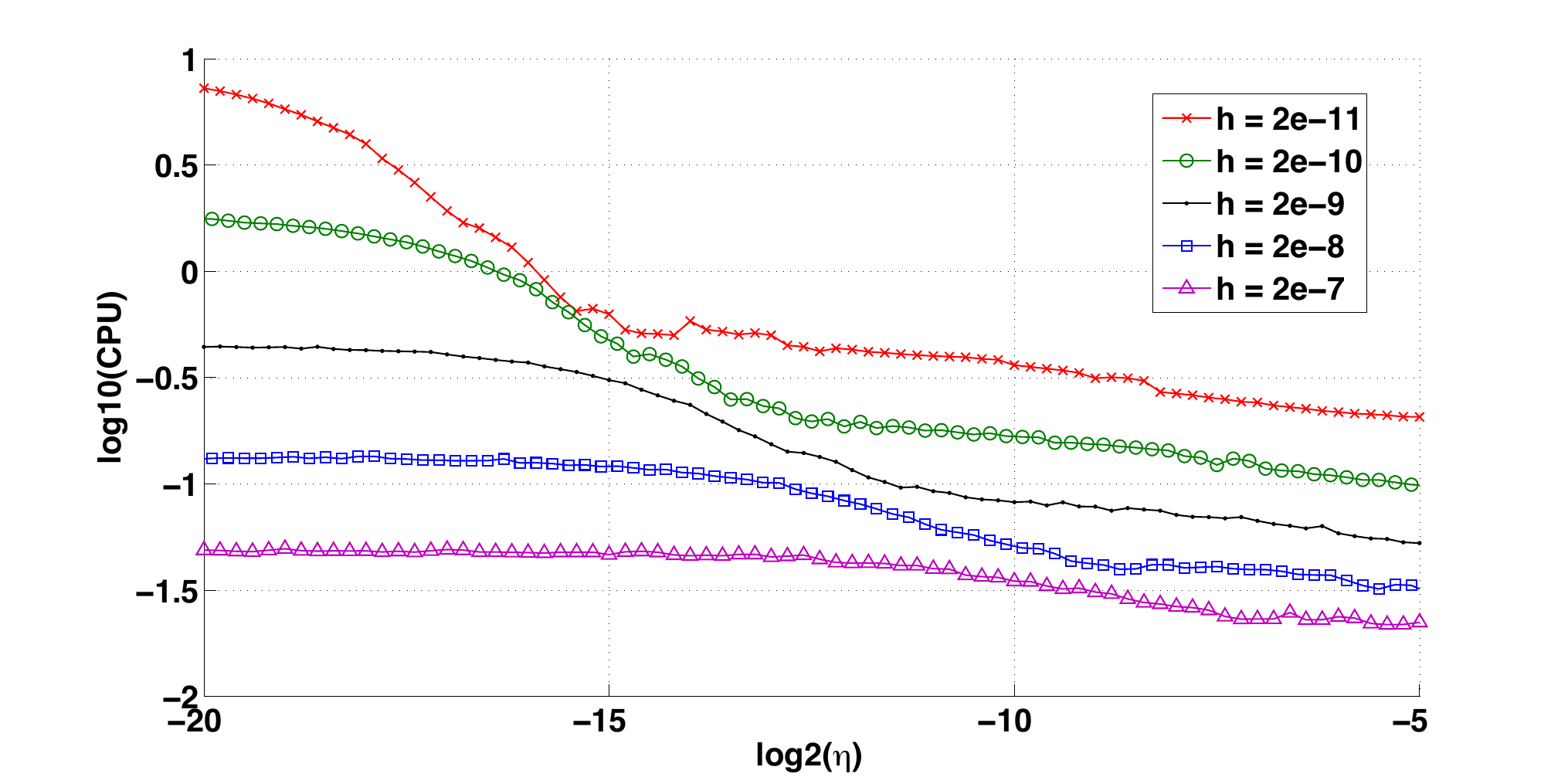}}}
  \end{center}
\caption{CPU time versus regularization parameter $\eta$ for different values of $h$ (time dependent potential) }
\label{etaCPUt}
\end{figure}

\begin{figure}[ht]
\begin{center}
\rotatebox{0}{\resizebox{!}{0.40\linewidth}{%
   \includegraphics{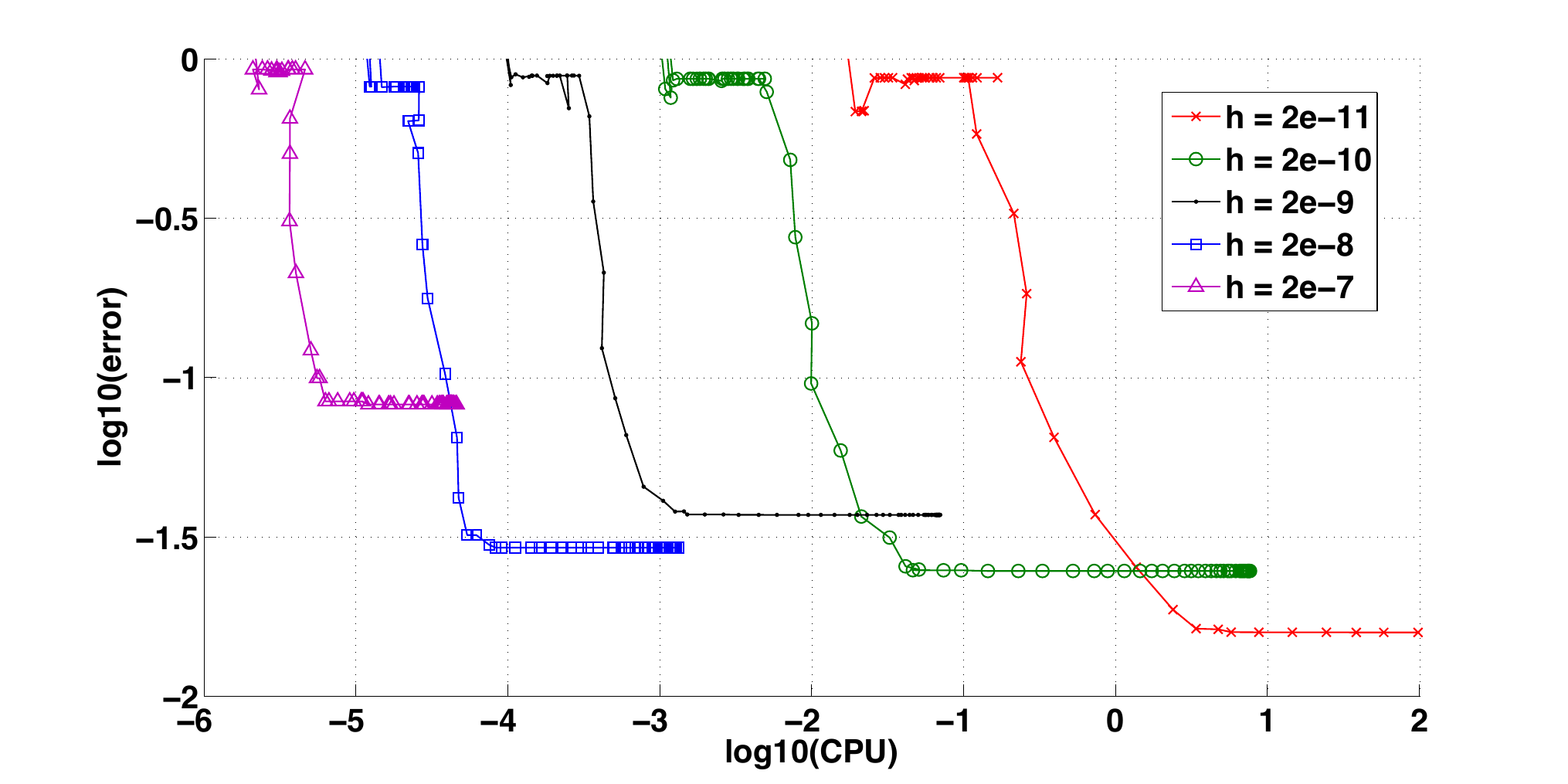}}}
  \end{center}
\caption{Error versus CPU time for different values of $h$ (time dependent potential) }
\label{CPUerrort}
\end{figure}

\appendix 
\section{Appendix: Proof of the a priori compactness Proposition \ref{lm2}}

We start with a lemma.

\begin{lm}
Assume that the hypothesis \textbf{(i)} and \textbf{(iii)} are satisfied. Recall that $L(t,x,v) = K^*(v) - V(t,x)$. 
For any $\Gamma> 1$, there exists a constant $\Gamma'$ such that for any $x,y\in \R^n$ and $T> 0$ and $t>1$, if $|x-y|/t<\Gamma$ and $\gamma$ that minimizes the quantity
$$\inf_{\substack{\gamma(0)=x \\ \gamma(t)=y          }} \int_0^t L\big(T+s,\gamma(s),\dot \gamma(s) \big) \d s,$$
then
$$\forall \, 0\leqslant a\leqslant a+1 \leqslant t, \quad |\gamma(a)-\gamma(a+1)|<\Gamma'.$$

\end{lm}

\begin{proof}
Without loss of generality, we will assume that $L$ is positive. Indeed, the potential $V$ is bounded, and adding a constant doesn't change the minimizers. In this case, and under the hypothesis \textbf{\em (i)}, there exists a nonnegative, increasing function $\alpha : [0,\infty) \to \R$  which verifies $\lim\limits_{t\to \infty}\alpha(t)=\infty$,  such that 
$$\forall t,x,v,\quad L(t,x,v)\geqslant \alpha (|v|)|v|.$$

The idea of the proof is that if $\gamma$ at some point has a great velocity, then it must be slow later. It is then better to ``slow down" the fast part and accelerate the ``slow" one.

First, we set some notations. For all   $(x,y,t,T) \in \R^n \times \R^n \times \R_+ \times \R$, let
$$
A_T^t(x,y)=\inf_{\substack{\gamma(0)=x \\ \gamma(t)=y          }} \int_0^t L\big(T+s,\gamma(s),\dot \gamma(s) \big) \d s,
$$
be the Lagrangian action.

We start by showing, that the superlinearity of $L$ implies a superlinearity of $A_T^t$. As already done a few times, we may bound the action by comparing with a straight line, using that $L$ is uniformly bounded on sets of the form $\R \times \R^n \times B(0,R)$, where $B(0,R)$ is the ball of radius $R > 0$ in $\R^n$:

\begin{eqnarray}
\label{eq:+}
A_T^t(x,y)&\leqslant& \int_0^t L\big(T+s,\frac{(t-s)x+sy}{t},\frac{y-x}{t} \big) \d s \nonumber\\
&\leqslant& t C^+\big(\frac{|x-y|}{t}\big) \max\Big(\frac{|x-y|}{t},1\Big), 
\end{eqnarray}
for some increasing function $z \mapsto C^+\big(z\big)$ defined on $\R^+$. Let $\gamma$ realizing the infimum, and set 
$$
\mathcal E = \Big\{s \in [0,t] \quad\mbox{such that}\quad | \dot\gamma(s) |\geqslant \frac{|x-y|}{2t}\Big\}. 
$$
Then we get, using that $L>0$:
\begin{multline}
A_T^t(x,y)=\int_0^t L\big(T+s,\gamma(s),\dot \gamma(s) \big) \d s  \\
\geqslant \int_{\mathcal E  }L\big(T+s,\gamma(s),\dot \gamma(s) \big) \d s   \\
\geqslant \alpha\big(\frac{|x-y|}{2t} \big) \int_{\mathcal E  }|\dot \gamma(s)| \d s \geqslant \alpha\big(\frac{|x-y|}{2t} \big)\frac{|x-y|}{2}.\label{eq:-}
\end{multline}
The last inequality comes from the fact that when $\gamma$ is going at speed less than     $|x-y|/2t$ for time $t$, it cannot travel more than $|x-y|/2$, therefore the integral is greater than $|x-y|/2$ by the triangular inequality.
In other terms, equations \eqref{eq:+} and \eqref{eq:-} state that there are two positive functions $C^+$ and $C^-$ which can be easily made increasing, such that
\begin{equation}\label{mat1}
 C^-\big(\frac{|x-y|}{t}\big) \frac{|x-y|}{t}\leqslant \frac{A_T^t(x,y)}{t} \leqslant     C^+\big(\frac{|x-y|}{t}\big) \max\Big(\frac{|x-y|}{t},1\Big).
\end{equation}
Moreover, thanks to the superlinearity of $L$ those functions are coercive.

Now consider 
\begin{itemize}
\item $\Gamma''>\Gamma$ such that $20C^+(\Gamma)<C^-(\Gamma'')$, 
\item $\Gamma'''>\Gamma''$ such that $\Gamma'''/\Gamma''\in \N^*$ and  $30C^+(20\Gamma'')<C^-(\Gamma''')$, 
\item and finally $\Gamma'>\Gamma'''$ such that $40\Gamma'''/\Gamma''<C^-(\Gamma')  / C^+(\Gamma) $. 
\end{itemize}Let us verify that $\Gamma'$ satisfies the requirements of our lemma. 

Assume by contradiction that for some $x,y\in \R^n$, $t,T\in \R_+$ such that $|x-y| < t\Gamma$ and $\gamma$ realizing the action $A_T^t(x,y)$, there is an $a\in[0,t-1]$ such that $|\gamma(a)-\gamma(a+1)| \geqslant \Gamma'$.  As $\gamma $ is a minimizer, we have (using that $L>0$ and $\Gamma > 1$)
\begin{multline}
\label{eq:prout}
 t \Gamma C^+(\Gamma)\geqslant A_T^t(x,y)    \geqslant    \int_a^{a+1} L\big(T+ s ,\gamma(s),\dot \gamma(s) \big) \d s      \\
 \geqslant A_{T+a}^1\big( \gamma(a),\gamma(a+1) \big) \geqslant  \Gamma'C^-(\Gamma').
 \end{multline}
 Hence we obtain
 \begin{equation}
 \label{eq:estt}
 t \geqslant \frac{\Gamma'C^-(\Gamma')}{\Gamma C^+(\Gamma)} \geqslant \frac{40\Gamma'''}{\Gamma''}, 
 \end{equation}
using the fact that $\Gamma' > \Gamma$ and the definition of $\Gamma'$.

 We now assume that $a<t/2$, the other case may be treated similarly. Let $b\in [a,a+1]$ be the smallest such that $|\gamma(a)-\gamma(b)|=\Gamma'''$, and consider the sequence 
 $$
 c_i= b+2i \frac{\Gamma'''}{\Gamma''}, \quad i \in \{0,\cdots,k\},
 $$ 
 where $k$ is greatest possible integer such that $c_k\leqslant t$. Note that using \eqref{eq:estt} and $a  \leqslant t/2$, we have $k\geqslant   9$, and that for $i \in \{0,\cdots,k\}$, we have $c_{i+1} - c_i = 2 \Gamma'''/\Gamma''$. 
 
We claim that there exists an $i_0 \in \{0,\cdots,k\}$ such that 
$$
\frac{|\gamma(c_{i_0})-\gamma(c_{i_0+1})|}{c_{i_0+1}-c_{i_0}}\leqslant \Gamma''.
$$
 Indeed, otherwise we would have, using \eqref{mat1}
 \begin{eqnarray*}
 A_T^t(x,y) &\geqslant& \sum_{i=0}^{k-1}\int_{c_{i}}^{c_{i+1}} L\big(T+ s ,\gamma(s),\dot \gamma(s) \big) \d s    \\
 &\geqslant & \sum_{i=0}^{k-1}\frac{|\gamma(c_{i_0})-\gamma(c_{i_0+1})|}{c_{i_0+1}-c_{i_0}} C^-(\Gamma'')\\
 &>&  \sum_{i=0}^{k-1}\frac{2\Gamma'''}{\Gamma''} \Gamma''C^-(\Gamma''). 
\end{eqnarray*}
By definition of $k$, we have that $(k+1)\times  2\Gamma'''/\Gamma'' \geqslant t/2 -1$ while using \eqref{eq:estt}, we have $t\geqslant 40 \Gamma'''/\Gamma'' \geqslant 40$  and $2\Gamma'''/\Gamma'' \leqslant t/18$. Hence we deduce that $k\times 2\Gamma'''/\Gamma''\geqslant t/3$. 
As $\Gamma''  \geqslant \Gamma$, the previous equation yields 
$$
A_T^t(x,y) > \frac{t}{3}\Gamma C^-(\Gamma'') \geqslant t\Gamma C^+(\Gamma), 
$$
which is absurd in view of \eqref{eq:prout}.

Now we find a contradiction by constructing a curve $\delta$ which has an action less than $\gamma$. Let $[c,d]=[c_{i_0},c_{i_0+1}]$. Recall that $ K := \Gamma'''/\Gamma''$ is an integer. We define the curve $\delta$ as follows:
\begin{itemize}
\item $\delta(s)=\gamma(s)$ if $s\in [0,a] \cup [d,t]$;
\item on $[a,b+ K ]$, $\delta$ coincides with the curve minimizing  $A_{T+a}^{b+ K -a}\big(\gamma(a),\gamma(b)\big)$;
\item on $[b+ K ,c+ K ]$, $\delta$ is the translate of $\gamma$ : $\delta(s)=\gamma(s- K )$;
\item on $[c+ K , d]$ (recall that $d=c+2 K $)  $\delta$ coincides with the curve minimizing $A_{T+c+ K }^{ K }\big(\gamma(c),\gamma(d)\big)$.
\end{itemize}
We now compute the difference of action between $\gamma$ and $\delta$, recalling that $L$ is $1$-periodic in time:
\begin{align*}
\int_0^{t} L\big(T+ s ,&\gamma(s),\dot \gamma(s) \big) \d s -\int_0^{t} L\big(T+ s ,\delta(s),\dot \delta(s) \big) \d s
\\
& = \int_a^{b} L\big(T+ s ,\gamma(s),\dot \gamma(s) \big) \d s +\int_c^{d} L\big(T+ s ,\gamma(s),\dot \gamma(s) \big) \d s\\
& \phantom{=} - \int_a^{b+K} L\big(T+ s ,\delta(s),\dot \delta(s) \big) \d s - \int_{c+K}^{d} L\big(T+ s ,\delta(s),\dot \delta(s) \big) \d s   
\\
&\geqslant \Gamma'''C^-(\Gamma''')+\frac{2\Gamma'''}{\Gamma''}\Gamma''C^-(\Gamma'')\\
&\phantom{=}  -\frac{\Gamma'''}{\Gamma''} \Gamma''C^+(\Gamma'') - \frac{\Gamma'''}{\Gamma''}(2\Gamma'')C^+(2\Gamma'') >0. 
\end{align*}
This contradicts the minimality of $\gamma$.

\end{proof}

\begin{rem}\rm
In the previous proof, we only used the fact that $L$ is periodic in time. 
In \cite{itu}, a similar result is proved when $L$ is periodic in space (instead of in time). The idea of the proof is the same except that, when constructing the curve $\delta$, instead of translating it in time (in third part of the construction), it is translated in space, while the ``fast" part of $\gamma$  between $a$ and $b$ is replaced by a geodesic (straight line) between $\gamma(a)$ and the closest  point from $\gamma(a)$ in the grid $\gamma(b)+\mathbb Z^n$.
\end{rem}

We now prove  lemma \ref{lm2}:

\begin{proof}[proof of lemma \ref{lm2}]
Recall now that $L$ is periodic both in time and in space and that its Euler--Lagrange flow is  complete.
As in the previous lemma, assume $L>0$. Let $\Gamma$ and $\Gamma'$ be as in the previous lemma, and $\gamma$ be a minimizer such that $|\gamma(0)-\gamma(t)|/t\leqslant \Gamma$. The curve $\gamma$ is then a trajectory of the Euler--Lagrange flow. Let moreover $0\leqslant a\leqslant a+1 \leqslant t$.
Finally, by superlinearity of $L$, let $A(1)$ be given by Equation \eqref{eq:suplin}, such that $L(t,x,v)\geqslant |v|-A(1)$. We therefore obtain that, with the notations used in the previous proof, 
\begin{multline*}
\int_0^1 |\dot\gamma (a+s)| \d s -A(1) \\Ê
\leqslant \int_0^1 L\big(T+a+s , \gamma(a+s), \dot \gamma (a+s) \big) \d s \leqslant  C^+(\Gamma')\Gamma'.
\end{multline*}
Therefore, there is at least one point $s_0\in [0,1]$ such that 
$$
|\dot \gamma (a+s_0)| \leqslant A(1)+C^+(\Gamma')\Gamma' :=D.
$$
By periodicity of the Lagrangian, and completeness of the Euler-Lagrange flow, there exists a constant $D'$ depending only on $D$, such that $|\dot \gamma |\leqslant D'$ on $[a+s_0-1,a+s_0+1] \cap [0,t]  \supset [a,a+1]$. Since $a$ is arbitrary, this finishes the proof.
\end{proof}

\section{Appendix: Proof of  Theorem \ref{approx2}} \label{BB}

\begin{proof}[proof of Theorem \ref{approx2}] 
 The Idea of the proof a rather common technique which consists in interchanging the minimizing paths between the continuous and the fully discrete semi--groups.

Let us denote by   $\gamma_{t} :[t,t+T]\to \R^{n}$ a minimizer of \eqref{eq:laxol}. 
Recall that the curve $\gamma_t(s)$ is $C^2$. Let us set $y := \gamma_t(t)$ and $x := \gamma_t(t + T)$. We have 
$$
T_t^Tu (x) = u(y) + \int_{t}^{t+T} L\big(s,\gamma_{t}(s),\dot{\gamma}_{t}(s)\big)\d s. 
$$
By superlinearity \eqref{eq:suplin}, this implies that 
$$
\int_{t}^{t+T} |\dot \gamma_t(s)| \d s  \leqslant T A(1) + |T_t^T u(x) - u(y)|. 
$$
Comparing with the trivial curve $\gamma \equiv x$ in the definition of the Lax--Oleinik semi--group, we have that 
\begin{equation}\label{ineg1}
T_t^Tu(x)  \leqslant \SNorm{u}{\infty} + T\big( B+K^*(0)\big)
\end{equation}
where $B$ is the constant in equation \eqref{eq:hyp1} (Hypothesis \textit{\textbf{(i)}}). Moreover, since $L$ is bounded below \big(meaning $L(t,x,v)\geqslant b$ for some constant $b$, for all $(t,x,v)$\big), clearly, the action of any curve defined for a time $T$ is greater than $Tb$ which implies immediately that
\begin{equation}\label{ineg2}
T_t^Tu(x)  \geqslant -\SNorm{u}{\infty} + Tb.
\end{equation}
 Hence, there exists a constant $B_1$ depending only on $L$ and $\SNorm{u}{\infty}$ such that 
\begin{equation}\label{pasdequoisenerver}
|x - y| =|\gamma_t(t+T) - \gamma_t(t)|  \leqslant \int_{t}^{t+T} |\dot \gamma_t(s)| \d s \leqslant B_1(1+T). 
\end{equation}
Remarking that $\gamma_t$ is also a minimizer of the action \eqref{eq:action} under the constraint $\gamma(t) = y$ and $\gamma(t+T) = x$, we can apply 
Proposition \ref{lm1} which shows that  there exists a constant $M_1=M\big(B_1(1+T),T\big)$ depending only on $T$, $L$ and $\SNorm{u}{\infty}$ such that 
\begin{equation}
\label{eq:apriori}
\forall\, s \in [t,t+T], \quad |\dot \gamma_t(s) |  \leqslant M_1. 
\end{equation}

Assume now that $N$ is an integer such that $N \tau \leqslant T$. 
For all $i = 0,\ldots,N$ we define 
$$
x_i= \e\left\lfloor \frac{1}{\e} \gamma_t(t_i)\right\rfloor
$$
where for $i = 0,\ldots,N$, $t_i = t + i \tau$ and where the function $\lfloor \cdot \rfloor$ is the floor function, coordinate by coordinate. 
With these points, we associate the continuous piecewise linear path $\lambda$ defined by

 \begin{equation}
 \label{eq:path}
\lambda(s) = x_i+(s-t_i)\frac{x_{i+1}-x_i}{\tau}, \quad \mbox{for}\quad s \in [t_{i},t_{i+1}]. 
 \end{equation}

By definition of the points $x_i$, we have 
\begin{equation}
\label{eq:sqn}
\forall\, i\in [0,N],\quad |x_i-\gamma_t(t_i)|=|\lambda(t_i)-\gamma_t(t_i)|\leqslant \e \sqrt{n}.
\end{equation}
Now, using the bound \eqref{eq:apriori}, we have for all $i = 0,\ldots,N-1$, 
$$
|x_{i+1} - x_i| \leqslant 2\e\sqrt n + \int_{t_i}^{t_{i+1}} |\dot \gamma_t(s)| \d s  \leqslant 2\e\sqrt n + \tau M_1. 
$$
But this inequality implies that for all $i = 0,\ldots,N-1$, 
$$
\forall\,s \in [t_i,t_{i+1}], \quad |\lambda(s) - x_i|  \leqslant 2\e\sqrt n + \tau M_1,
$$
while $|\gamma_t(s) - \gamma_t(t_i)|  \leqslant \tau M_1$ upon using \eqref{eq:apriori}. 
Hence we get 
\begin{equation}
\label{eq:t1}
\forall\,s \in [t_i,t_{i+1}], \quad |\lambda(s) - \gamma_t(s)|  \leqslant 3\e \sqrt{n}+ 2\tau M_1. 
\end{equation}
Moreover, we have for $s, \sigma \in [t_{i},t_{i+1}]$, 
$$
|\dot \gamma_t(\sigma) - \dot \gamma_t(s)|  \leqslant \tau C,
$$
upon using \eqref{eq:2deriv}. Hence for $s \in [t_i,t_{i+1}]$, we have
$$
|\gamma_t(t_{i+1}) - \gamma_t(t_{i}) - \tau \dot \gamma_t(s)| \leqslant  \int_{t_i}^{t_{i+1}} |\dot \gamma_t(\sigma) - \dot \gamma_t(s)| \d \sigma \leqslant \tau^2 C,
$$
and hence for all $s \in [t_i,t_{i+1}]$
$$
\left |\frac{\gamma_t(t_{i+1}) - \gamma_t(t_{i})}{\tau} - \dot \gamma_t(s)\right| \leqslant \tau C. 
$$
Using \eqref{eq:sqn}, we obtain easily that for all $i = 0,\ldots,N-1$, 
\begin{equation}
\label{eq:t2}
\forall\, s \in [t_i,t_{i+1}], \quad | \dot \lambda(s)  - \dot \gamma_t(s)|  \leqslant \tau C + \frac{2\e}\tau\sqrt{n} .
\end{equation}
Note that using \eqref{eq:anticfl} and \eqref{eq:apriori}, the previous equation implies that for all $i = 0,\ldots,N-1$, 
\begin{equation}
\label{eq:bornd}
\forall\, s \in [t_i,t_{i+1}],\quad | \dot \lambda(s)|  \leqslant M_2
\end{equation}
for some constant $M_2=\tau_0C+h_0\sqrt{n}+M_1$ independent of $\e$ and $\tau$. 

Now by definition of $\lambda$, we have 
\begin{multline}
\label{eq:crux}
\Big|\int_0^{N\tau}L\big(s,\gamma_t(s),\dot{\gamma}_t(s)\big) \d s -\int_0^{N\tau}L\big(s,\lambda(s),\dot\lambda(s)\big) \d s\Big|\\
\leqslant\sum_{i = 0}^{N-1}\int_{t_i}^{t_{i+1}}\Big|L\big(s,\gamma_t(s),\dot{\gamma}_t(s)\big) -L\big(s,\lambda(s),\dot\lambda(s)\big)\Big|\d s.
\end{multline}
Using \eqref{eq:hyp1} (coming from Hypothesis \textbf{\em (ii)}), the fact that $K^*$ is $C^2$ and the bounds \eqref{eq:apriori} and \eqref{eq:bornd}, there exists a constant $M_3$, depending on $L$, $M_1$ and $M_2$, such that 
the previous error term is bounded by 
$$
M_3\sum_{i = 0}^{N-1}\int_{t_i}^{t_{i+1}} \Big( |\gamma_t(s) - \lambda(s)| + | \dot\gamma_t(s) - \dot \lambda(s) |\Big) \d s. 
$$
Using \eqref{eq:t1} and \eqref{eq:t2}, this shows that there exists a constant $M_4$ independent on $\e$ and $\tau$, such that 
\begin{equation}
\label{eq:jkl}
\left|\int_t^{t + N\tau}L\big(s,\gamma_{t}(s),\dot{\gamma}_{t}(s)\big)\d s-\int_0^{N\tau}L\big(s,\lambda(s),\dot\lambda(s)\big) \d s\right|\leqslant M_4\big(\frac{\e}{\tau}+\tau\big),
\end{equation}
where we used the fact that $\e \leqslant  \tau_0 \e / \tau$. 

Finally, the term we wish to estimate is 
\begin{multline*}
\Big|\int_t^{t+N\tau}   L\big(s,\gamma_t(s),\dot{\gamma}_t(s)\big) \d s-\sum_{i=0}^{N-1} \kappa_{t,\e}^{\tau}(x_{i},x_{i+1})\Big|\\
\leqslant \sum_{i = 0}^{N-1}\int_{t_i}^{t_{i+1}}\Big|L\big(s,\gamma_t(s),\dot{\gamma}_t(s)\big) -L\big(s,\lambda(s),\dot\lambda(s)\big)\Big|\d s\\
+ \sum_{i = 0}^{N-1}\int_{t_i}^{t_{i+1}}\Big|L\big(s,\lambda(s),\dot\lambda(s)\big) -L\big(t_i,\lambda(t_i),\dot\lambda(t_i)\big)\Big|\d s.
\end{multline*}

To bound the second term, we observe first that for $s \in [t_i,t_{i+1}]$, the derivative $\dot\lambda(s) = (x_{i+1} - x_i)/\tau$ does not depend on $s$. Hence using \eqref{eq:hyp1} and \eqref{eq:bornd} the function 
$$
[t_i,t_{i+1}] \ni s \mapsto L\big(s,\lambda(s),\dot\lambda(s)\big)
$$
is $\mathcal{C}^1$ with uniformly bounded derivative. Thus we obtain that there exists a constant $M_5$ such that 
$$
\Big| L\big(s,\lambda(s),\dot\lambda(s)\big) -L\big(t_i,\lambda(t_i),\dot\lambda(t_i)\big)\Big|  \leqslant M_5(s - t_i). 
$$

This proves that \big(compare \eqref{eq:jkl}\big)
\begin{equation}
\label{eq:jkl2}
\left|\int_t^{t + N\tau}L\big(s,\gamma_{t}(s),\dot{\gamma}_{t}(s)\big)\d s-\sum_{i=0}^{N-1} \kappa_{t_i,\e}^{\tau}(x_i,x_{i+1})\right|\leqslant M_6\big(\frac{\e}{\tau}+\tau\big),
\end{equation}
for some constant $M_6$ independent of $\e$ and $\tau$. 

Now by definition of $\Tc^N$, we have using \eqref{eq:sqn} and the Lipschitz nature of $u$, 
\begin{eqnarray}
\label{eq:klm}
\Tc^N (u|_{G_{\e}})(x)  &\leqslant& u(x_0) + \sum_{i = 0}^{N-1} \kappa_{t_i,\e}^{\tau}(x_{i},x_{i+1})\nonumber\\
&\leqslant& (T^{N\tau}_t u)|_{G_{\e}}(x)  + \big|u(x_0) - u\big(\gamma_t(0)\big)\big| + M_4\big(\frac{\e}{\tau}+\tau\big)\nonumber\\
&\leqslant& (T^{N\tau}_t u)|_{G_{\e}}(x) + M\big(\frac{\e}{\tau}+\tau\big)
\end{eqnarray}
for some constant $M$ independent on $\e$ and $\tau$.  This proves a first inequality in the estimate \eqref{eq:conver} with the notations of the Theorem. 

To prove the reverse inequality, let us fix $x \in G_{\e}$. We consider a sequence $y_i$, $i = 0,\ldots, N$ with $y_N = x$ and 
\begin{equation}
\label{eq:raslebol}
\Tc^N (u|_{G_{\e}})(x) = u(y_0) + \sum_{i = 0}^{N-1} \kappa_{t_i,\e}^{\tau}(y_{i},y_{i+1}), 
\end{equation}
and we define the curve 
$$
\eta(s) = y_i+(s-t_i)\frac{y_{i+1}-y_i}{\tau}, \quad \mbox{for}\quad s \in [t_{i},t_{i+1}].
$$
Note that in a similar manner to what we did to prove the inequalities \ref{ineg1} and \ref{ineg2},  using the fact that $u$ is bounded, and comparing with the trivial sequence made of a constant point (with $\eqref{eq:hyp1}$) on the one hand, and the fact that $L$, hence the $\kappa_{t_i,\e}^{\tau}$ are bounded below on the second hand  show that there exists a constant $D_1$ such that 
$$
\SNorm{\Tc^N (u|_{G_{\e}})}{\infty}  \leqslant D_1. 
$$
By superlinearity of $L$ (and of the $\kappa_{t_i,\e}^{\tau}$)  and using again the fact that $u$ is bounded, we thus see, as in \ref{pasdequoisenerver} that there exists a constant $D_2$ such that for all $i = 0,\ldots,N-1$, 
$$
\left|\frac{y_{i+1} - y_{i}}{\tau} \right|  \leqslant D_2, 
$$
which in turn implies that 
$$
\forall\, s \in [t_i,t_{i+1}], \quad |\eta(s) - \eta(t_i)|  \leqslant \tau D_2.
$$
 As the derivative of $\eta(s)$ with respect to $s$ is uniformly bounded by $D_2$ and  constant on the time intervals $[t_i,t_{i+1}]$, and as $L$ is $\mathcal{C}^1$ with uniformly bounded derivative on $\R \times \R^n\times B(0,D_2)$, we obtain 
\begin{equation}
\label{eq:marre}
\left| \sum_{i = 0}^{N-1} \kappa_{t_i,\e}^{\tau}(y_{i+1},y_i) - \int_{0}^{N\tau} L\big(s,\eta(s),\dot \eta(s)\big) \d s \right|Ê \leqslant \tau D_3
\end{equation}
for some constant $D_3$. Using the definition of the exact semi--group, we thus have  
\begin{eqnarray*}
(T^{N\tau}_t u)|_{G_{\e}}(x) &\leqslant& u\big(\eta(t_N)\big) + \int_{0}^{N\tau} L\big(s,\eta(s), \dot \eta(s) \big) \d s \\
&\leqslant& \Tc^N (u|_{G_{\e}})(x) + \tau D_3
\end{eqnarray*}
upon using \eqref{eq:raslebol} and \eqref{eq:marre}. This proves the result. 
\end{proof}

\begin{proof}[proof of Theorem \ref{approx3}]
In the proof of  Theorem \ref{approx2}, equation \eqref{pasdequoisenerver} then gives using Proposition \ref{lm2} (with $T  \geqslant T_0 > 1$) that the constant $M_1$ defined in \eqref{eq:apriori} does not depend on $T = N\tau$ and depends in fact only on $T_0$. It then follows that $M_2$ and $M_3$ also are independent  of $T=N\tau$, while $M_4$ is proportional to the time of integration, that is $N\tau$.
 The rest of the proof can then be carried on giving the result.
\end{proof}

\bibliography{xou}

\begin{thebibliography}{BCOQ92}

\bibitem[Abg96]{abgrall}
R.~Abgrall.
\newblock Numerical discretization of the first-order {H}amilton-{J}acobi
  equation on triangular meshes.
\newblock {\em Comm. Pure Appl. Math.}, 49(12):1339--1373, 1996.

\bibitem[AGL08]{akian}
M.~Akian, S.~Gaubert, and A.~Lakhoua.
\newblock The max-plus finite element method for solving deterministic optimal
  control problems: basic properties and convergence analysis.
\newblock {\em SIAM J. Control Optim.}, 47(2):817--848, 2008.

\bibitem[BB07]{Be}
P.~Bernard and B.~Buffoni.
\newblock Weak {KAM} pairs and {M}onge-{K}antorovich duality.
\newblock In {\em Asymptotic analysis and singularities---elliptic and
  parabolic {PDE}s and related problems}, volume~47 of {\em Adv. Stud. Pure
  Math.}, pages 397--420. Math. Soc. Japan, Tokyo, 2007.

\bibitem[BCOQ92]{BCOQ}
F.~Baccelli, G.~Cohen, G.Y. Olsder, and J.P. Quadrat.
\newblock {\em Synchronization and linearity}.
\newblock Wiley, 1992.

\bibitem[BJ05]{barlesjakobsen}
G.~Barles and E.~R. Jakobsen.
\newblock Error bounds for monotone approximations schemes for
  {H}amilton-{J}acobi-{B}ellman equations.
\newblock {\em SIAM J. Numer. Anal.}, 43:540--558, 2005.

\bibitem[BJT08]{BJT2008}
A.~Bouillard, L.~Jouhet, and E.~Thierry.
\newblock Computation of a (min,+) multi-dimensional convolution for end-to-end
  performance analysis.
\newblock In {\em Proceedings of Valuetools'2008}, 2008.

\bibitem[BR04]{BJ}
Patrick Bernard and Jean-Michel Roquejoffre.
\newblock Convergence to time-periodic solutions in time-periodic
  {H}amilton-{J}acobi equations on the circle.
\newblock {\em Comm. Partial Differential Equations}, 29(3-4):457--469, 2004.

\bibitem[Bre89]{brenier}
Y.~Brenier.
\newblock Un algorithme rapide pour le calcul de transofrm\'eesde
  {L}egendre-{F}enchel discr\`etes.
\newblock {\em C. R. Acad. Sci. Paris. S\'er. I Math.}, 308(20):587--589, 1989.

\bibitem[BS91]{BaSo}
G.~Barles and P.~E. Souganidis.
\newblock Convergence of approximation schemes for fully nonlinear second order
  equations.
\newblock {\em Asymptotic Anal.}, 4(3):271--283, 1991.

\bibitem[BT08]{toolbox}
A.~Bouillard and E.~Thierry.
\newblock An algorithmic toolbox for network calculus.
\newblock {\em Discrete Event Dynamic Systems}, 18(1):3--49, 2008.

\bibitem[CISM00]{cis}
G.~Contreras, R.~Iturriaga, and H.~Sanchez-Morgado.
\newblock {W}eak solutions of the {H}amilton-{J}acobi equation for time
  periodic {L}agrangians.
\newblock {\em preprint}, 2000.

\bibitem[CL84]{Crandall}
M.~G. Crandall and P.-L. Lions.
\newblock Two approximations of solutions of {H}amiltonÐ{J}acobi equations.
\newblock {\em Math. Comp.}, 43:1--19, 1984.

\bibitem[Fal87]{falcone2}
M.~Falcone.
\newblock A numerical approach to the infinite horizon problem of deterministic
  control theory.
\newblock {\em Appl. Math. and Optim.}, 15:1--13, 1987.
\newblock {\em Corrigenda} Appl. Math. and Optim. Vol 23 (1991), 213--214.

\bibitem[Fat97]{FaKAM}
A.~Fathi.
\newblock Th\'eor\`eme {KAM} faible et th\'eorie de {M}ather sur les syst\`emes
  lagrangiens.
\newblock {\em C. R. Acad. Sci. Paris S\'er. I Math.}, 324(9):1043--1046, 1997.

\bibitem[Fat98]{faconver}
A.~Fathi.
\newblock Sur la convergence du semi-groupe de {L}ax-{O}leinik.
\newblock {\em C. R. Acad. Sci. Paris S\'er. I Math.}, 327(3):267--270, 1998.

\bibitem[Fat05]{Fathi}
A.~Fathi.
\newblock Weak {KAM} {T}heorem in {L}agrangian {D}ynamics, preliminary version,
  {P}isa, 16 f\'evrier 2005.

\bibitem[FF02]{falcone}
M.~Falcone and R.~Ferreti.
\newblock Semi-{L}agrangian schemes for {H}amilton-{J}acobi equations, discrete
  representation formulae and {G}odunov methods.
\newblock {\em J. Comput. Phys.}, 175:559--575, 2002.

\bibitem[FM07]{FaMa}
A.~Fathi and E.~Maderna.
\newblock Weak {KAM} theorem on non compact manifolds.
\newblock {\em NoDEA}, 14(1):1--27, 2007.

\bibitem[HLW06]{hairer06gni}
E.~Hairer, C.~Lubich, and G.~Wanner.
\newblock {\em Geometric Numerical Integration. Structure-Preserving Algorithms
  for Ordinary Differential Equations, Second edition}.
\newblock Springer Series in Computational Mathematics 31. Springer, Berlin,
  2006.

\bibitem[Itu96]{itu}
R.~Iturriaga.
\newblock Minimizing measures for time-dependent {L}agrangians.
\newblock {\em Proc. London Math. Soc. (3)}, 73(1):216--240, 1996.

\bibitem[JKR01]{splitting}
E.~R. Jakobsen, K.~H. Karlsen, and N.~H. Risebro.
\newblock On the convergence rate of operator splitting for {H}amilton-{J}acobi
  equations with source terms.
\newblock {\em SIAM J. Numer. Anal.}, 39(2):499--518, 2001.

\bibitem[JP00]{jiangpeng}
G.~Jiang and D.-P. Peng.
\newblock Weighted {ENO} schemes for {H}amilton-{J}acobi equations.
\newblock {\em SIAM J. Sci. Comput.}, 21:2126--2143, 2000.

\bibitem[JS96]{jiangshu}
G.-S. Jiang and C.-W. Shu.
\newblock Efficient implementation of weighted {E}{N}{O} schemes.
\newblock {\em J. Comput. Phys.}, 126:202--228, 1996.

\bibitem[JX98]{jinxin}
S.~Jin and Z.~Xin.
\newblock Numerical passage from systems of conservation laws to
  {H}amilton-{J}acobi equations.
\newblock {\em SIAM J. Numer. Anal.}, 35:2385--2404, 1998.

\bibitem[LBT01]{Leboudec}
J.-Y. Le~Boudec and P.~Thiran.
\newblock {\em Network Calculus: A Theory of Deterministic Queuing Systems for
  the Internet}, volume LNCS 2050.
\newblock Springer-Verlag, revised version 4, may 10, 2004 edition, 2001.

\bibitem[Lio82]{Lions}
P.-L. Lions.
\newblock {\em Generalized solutions of {H}amilton-{J}acobi equations},
  volume~69 of {\em Research Notes in Mathematics}.
\newblock Pitman (Advanced Publishing Program), Boston, Mass., 1982.

\bibitem[LPV87]{LPV}
P.-L. Lions, G.~Papanicolaou, and S.R.S. Varadhan.
\newblock Homogenization of {H}amilton-{J}acobi equation.
\newblock unpublished preprint, 1987.

\bibitem[LR04]{leimreich}
B.~Leimkuhler and S.~Reich.
\newblock {\em Simulating Hamiltonian dynamics}.
\newblock Cambridge Monographs on Applied and Computational Mathematics 14.
  Cambridge University Press, Cambridge, 2004.

\bibitem[LS95]{lionssouganidis}
P.-L. Lions and P.~E. Souganidis.
\newblock Convergence of {MUSCL} and filtered schemes for scalar conservation
  laws and {H}amilton-{J}acobi equations.
\newblock {\em Numer. Math.}, 69:441--470, 1995.

\bibitem[LT01]{lintadmor}
C.-T. Lin and E.~Tadmor.
\newblock ${L}^1$-stability and error estimates for approximate
  {H}amilton-{J}acobi solutions.
\newblock {\em Numer. Math.}, 87:701--735, 2001.

\bibitem[Luc97]{L97}
Y.~Lucet.
\newblock Faster than the fast {L}egendre transform, the linear-time {L}egendre
  transform.
\newblock {\em Numer. Algorithms}, 16(2):171--185, 1997.

\bibitem[Mat91]{Ma2}
J.~N. Mather.
\newblock Action minimizing invariant measures for positive definite
  {L}agrangian systems.
\newblock {\em Math. Z.}, 207(2):169--207, 1991.

\bibitem[OS88]{oshersethian}
S.~Osher and J.~Sethian.
\newblock Fronts propagating with curvature dependent speed: algorithms based
  on {H}amilton-{J}acobi formulations.
\newblock {\em J. Comput. Phys.}, 79:12--49, 1988.

\bibitem[OS91]{oshershu}
S.~Osher and C.-W. Shu.
\newblock High-order essentially nonoscillatory schemes for {H}amiltonÐ{J}acobi
  equations.
\newblock {\em SIAM J. Numer. Anal.}, 28:907--922, 1991.

\bibitem[Ror06]{rorro}
M.~Rorro.
\newblock An approximation scheme for the effective {H}amiltonian and
  applications.
\newblock {\em Appl. Numer. Math.}, 56(9):1238--1254, 2006.

\bibitem[Sog13]{soga}
K.~Soga.
\newblock Stochastic and variational approach to the {L}ax-{F}riedrichs scheme.
\newblock {\em Math. Comp.}, in press, 2013.

\bibitem[Sou85]{souganidis}
P.~E. Souganidis.
\newblock Approximation schemes for viscosity solutions of {H}amilton-{J}acobi
  equations.
\newblock {\em J. Differential Equations}, 59:1--43, 1985.

\bibitem[Zav12]{Za}
M.~Zavidovique.
\newblock Strict subsolutions and {M}a\~n\' e potential in discrete weak {KAM}
  theory.
\newblock {\em {C}omment. {M}ath. {H}elv.}, 87(1):1--39, 2012.

\end{thebibliography}
\bibliographystyle{alpha}
\end{document}